\documentclass[twoside,11pt]{article}

\usepackage{jmlr2e}
\usepackage{comment}
\usepackage{amsmath}
\usepackage{times}
\usepackage{subcaption}
\usepackage{graphics}
\usepackage{bbold}
\usepackage{algorithm2e}
\usepackage{hyperref}
\hypersetup{
	bookmarks=true,         
	colorlinks=true,       
	linkcolor=red,          
	citecolor=green,        
	filecolor=magenta,      
	urlcolor=cyan           
}
\usepackage{tikz}

\newlength\Colsep
\newcommand{\inr}[1]{\bigl< #1 \bigr>}
\newcommand{\cF}{\mathcal{F}}
\newcommand{\cK}{\mathcal{K}}
\newcommand{\cO}{\mathcal{O}}
\newcommand{\cX}{\mathcal{X}}
\newcommand{\cB}{\mathcal{B}}
\newcommand{\cY}{\mathcal{Y}}
\newcommand{\cI}{\mathcal{I}}
\newcommand{\bR}{\mathbb{R}}

\newcommand{\R}{\mathbb{R}}
\newcommand{\E}{\mathbb{E}}
\newcommand{\cL}{\mathcal{L}}
\newcommand{\norm}[1]{\left\|#1\right\|}%

\newcommand{\MOM}[2]{\text{MOM}_{#1}\left[#2\right]}
\newcommand{\cro}[1]{\left[#1\right]}
\newcommand{\set}[1]{\left\{#1\right\}}

\newcommand{\cA}{\mathcal{A}}

\DeclareMathOperator*{\argmin}{argmin}

\newtheorem{Theorem}{Theorem}
\newtheorem{Assumption}{Assumption}
\newtheorem{Definition}{Definition}
\newtheorem{Lemma}{Lemma}
\newtheorem{Proposition}{Proposition}

\newtheorem{Remark}{Remark}

\firstpageno{1}

\usepackage{lastpage}
\jmlrheading{21}{2020}{1-\pageref{LastPage}}{7/19; Revised
10/20}{11/20}{19-585}{Chinot Geoffrey, Lecu{\'e} Guillaume and Lerasle Matthieu}
\ShortHeadings{Robust high dimensional learning for Lipschitz and convex losses}{Chinot, Lecu{\'e} and Lerasle}

\begin{document}
	
	\title{Robust high dimensional learning for Lipschitz and convex losses. 
}
	
	\author{\name Chinot Geoffrey \email geoffrey.chinot@stat.math.ethz.ch \\
		\addr Department of Statistics\\
		ETH Zurich\\
		Rämistrasse 101, 8092 Zurich, Switzerland
		\AND
		\name Lecu{\'e} Guillaume \email guillaume.lecue@ensae.fr \\
		\addr Department of Statistics\\
		ENSAE CREST\\
		5 avenue Henry Le Chatelier 91120 Palaiseau, France
		\AND
		\name Lerasle Matthieu \email matthieu.lerasle@ensae.fr \\
		\addr Department of Statistics\\
		ENSAE CREST\\
		5 avenue Henry Le Chatelier 91120 Palaiseau, France
	}
	
	\editor{Nicolas Vayatis}
	
	\maketitle
	
	\begin{abstract}
	We establish risk bounds for Regularized Empirical Risk Minimizers (RERM) when the loss is Lipschitz and convex and the regularization function is a norm. In a first part, we obtain these results in the i.i.d. setup under subgaussian assumptions on the design. In a second part, a more general framework where the design might have heavier tails and data may be corrupted by outliers both in the design and the response variables is considered. In this situation, RERM performs poorly in general. We analyse an alternative procedure based on median-of-means principles and called ``minmax MOM''. We show optimal subgaussian deviation rates for these estimators in the relaxed setting. The main results are meta-theorems allowing a wide-range of applications to various problems in learning theory. To show a non-exhaustive sample of these potential applications, it is applied to classification problems with logistic loss functions regularized by LASSO and SLOPE, to regression problems with Huber loss regularized by Group LASSO and Total Variation. Another advantage of the minmax MOM formulation is that it suggests a systematic way to slightly modify descent based algorithms used in high-dimensional statistics to make them robust to outliers \cite{lecue2017robust}. We illustrate this principle in a Simulations section where a `` minmax MOM'' version of classical proximal descent algorithms are turned into  robust to outliers algorithms.  
	\end{abstract}
	

\section{Introduction}
Regularized empirical risk minimizers (RERM) are standard estimators in high dimensional classification and regression problems. They are solutions of  minimization problems of a regularized empirical risk functions for a given loss and regularization functions. 
In regression, the quadratic loss of linear functionals regularized by the $\ell_1$-norm (LASSO) \cite{tibshirani1996regression} is probably the most famous example of RERM, see for example \cite{MR2829871, MR2807761,MR3307991} for overviews.
Recent results and references, including more general regularization functions can be found, for example in \cite{LM_sparsity, bellec2017towards,MR3025128,MR3138795,MR3236867}. 
RERM based on the quadratic loss function are highly unstable when data have heavy-tails or when the dataset has been corrupted by outliers.
These problems have attracted a lot of attention in robust statistics, see for example \cite{huber2011robust} for an overview. 
By considering alternative losses, one can efficiently solve these problems when heavy-tails or corruption happen in the output variable $Y$. 
There is a growing literature analyzing performance of some of these alternatives in learning theory.
In regression problems, among others, one can mention the $L_1$ absolute loss \cite{shalev2011stochastic}, the Huber loss~\cite{Fan, MR3852659} and the quantile loss \cite{pierre2017estimation} that is popular in finance and econometrics. 
In classification, besides the $0/1$ loss function which is known to lead to computationally intractable RERM, the logistic loss and the hinge loss are among the most popular convex surrogates \cite{MR2051001,MR2268032}. 
Quantile, $L_1$, Huber loss functions for regression and Logistic, Hinge loss functions for classification are all Lipschitz and convex loss functions (in their first variable, see Assumption~\ref{assum:lip} for a formal definition). 
This remark motivated \cite{pierre2017estimation} to study systematically RERM based on Lipschitz loss functions.
A remarkable feature of Lipschitz losses proved in \cite{pierre2017estimation} is that optimal results can be proved with almost no assumption on the response variable $Y$.


This paper is built on the approach initiated in \cite{ChiLecLer:2018}.
Compared with \cite{pierre2017estimation}, the approach of \cite{ChiLecLer:2018} improves the results by  deriving risk bounds depending on a localized complexity parameters rather than global ones and by considering a more flexible setting where a global Bernstein condition is relaxed into a local one, see Assumption~\ref{assum:fast_rates} and the following discussion for details.
The paper \cite{ChiLecLer:2018} only considers estimators that are not regularized and that can therefore only be efficient in small dimensional settings. 

The first main result of this paper is a high dimensional extension of the results in \cite{ChiLecLer:2018} that is achieved by analyzing estimators (based on the empirical risk or a Median-of-Means version) regularized by a norm. The main results are two meta-theorem allowing to study a broad range of estimators including LASSO, SLOPE, group LASSO, total variation and their minmax MOM version. Section~\ref{sec:app} provides applications of the main results to some examples among these.

While RERM is studied without assumption on the output variables, somehow strong, albeit classical, hypotheses are granted on the design $X$ in our first main result.
We assume actually in this analysis subgaussian assumptions on the input variables as in~\cite{pierre2017estimation}. 
The necessity of this assumption to derive optimal exponential deviation bounds for RERM is not surprising as RERM have downgraded performance when the design is heavy tailed (see \cite{mendelson2014learning} or \cite{ChiLecLer:2018} for instance). 

In a second part, we study an alternative to RERM in a framework with less stringent assumptions on the data. These estimators are based on the Median-Of-Means (MOM) principle \cite{MR702836, MR762855, MR855970, MR1688610} and the minmax approach \cite{MR2906886,MR3595933}. They are called minmax MOM estimators as in \cite{lecue2017robust}. A non-regularized version of these estimators was analyzed in \cite{ChiLecLer:2018}. The second main and most important result of the paper shows that minmax MOM estimators achieve optimal subgaussian deviation bounds in the relaxed setting where RERM perform poorly because of outliers and heavy-tailed data.  This result is obtained under a local Bernstein condition as for the RERM.
It allows to derive fast rates of convergence in a large set of applications where typically, subgaussian assumptions on the design $X$ are replaced by moment assumptions.
%
%
Minmax MOM estimators are then analysed without the local Bernstein condition.
Oracle inequalities holding with exponentially large probability are proved in this case.  
	Compared with results under Bernstein's assumption, an extra variance term appears in the convergence rate. 
	This extra term typically would yield to slow rates of convergence in the applications, which are known to be minimax in the case where no Bernstein assumption holds.
	However, the variance term disappears under the Bernstein's condition, which shows that fast rates can be recovered from the general results.
	In addition, all results on minmax MOM estimators, both with or without Bernstein condition, are shown in the ``$\cO\cup\cI$'' framework  -- where $\cO$ stands for ``outliers'' and $\cI$ for ``informative''-- see Section~\ref{sec:OUI} or~\cite{lecue2017learning,lecue2017robust} for details.
	In this framework, all assumptions (such as the Bernstein's condition) are granted on ``inliers" $(X_i,Y_i)_{i\in \cI}$.
	These inliers may have different distributions but the oracles of these distributions should match.
	On the other hand, no assumption are granted on outliers $(X_i,Y_i)_{i\in \cO}$, which is to the best of our knowledge the strongest form of aggressive/adversarial outliers (it includes, in particular, Huber's $\epsilon$-contamination setup).
	The minmax MOM estimators perform well in this setting, it means that the accuracy of their predictions is not downgraded by the presence of outliers in the dataset.
	Mathematically, this robustness is not surprising as it is a byproduct of the median step used in the MOM principle.
	However, in practice, it is an important advantage of MOM estimators compared to RERM. 
	
	The main results on minmax MOM estimators are also meta-theorems that can be applied to the same examples as RERM.
	Each of these examples provide a new (to the best of our knowledge) estimator that reach performance that RERM could not typically achieve.
	For example, when the class of classifiers/regressors is the class of linear functions on $\R^p$, minmax MOM estimators have a risk bounded by the minimax rate with optimal exponential probability of deviation even if the inputs $X$ only satisfy weak moment assumptions and/or have been corrupted by outliers.
	These applications are also discussed in Section~\ref{sec:app}.

	Finally, in Section~\ref{sec:Simu}, we consider the modification of standard algorithms suggested by the minmax MOM formulation introduced in \cite{lecue2017robust} to construct robust algorithms.  


 
%

The  paper is organized as follows. Section~\ref{sec:setting} presents the formal setting. Section~\ref{sec_ERM} presents results for RERM and Section~\ref{sec:minmaxMOM} those for minmax MOM estimators under a local Bernstein condition and in Section~\ref{sec_wtb} without this condition. Section~\ref{sec:app} details several examples of applications of the main results. A short simulation study illustrating our theoretical findings is presented in Section~\ref{sec:Simu}. The proofs are postponed to Sections~\ref{sec:proofs}-~\ref{sec:MOM_wtb}.

\section{Mathematical background and notations}\label{sec:setting}
Let $(\mathcal{Z},\mathcal{A},P)$ denote a probability space, where $\mathcal{Z}=\mathcal{X}\times\mathcal{Y}$ is a product space such that $\mathcal{X}$ denotes a measurable space of inputs and $\mathcal{Y}\subset\R$ is the set of values taken by the outputs. 
Let $Z=(X,Y)$ denote a random variable taking values in $\mathcal{Z}$ with distribution $P$ and let $\mu$ denote the marginal distribution of the design $X$.

Let $\overline{\mathcal{Y}}\subset \R$ denote a convex set such that $\cY\subset\overline{\mathcal{Y}}$ and let $F$ denote a class of functions $f:\mathcal{X}\to\overline{\mathcal{Y}}$.
The set $\overline{\mathcal{Y}}$ is typically the co,vex hull of $\cY$.  As such, it will always contain $\cY$.
Let $\ell:\overline{\mathcal{Y}}\times \mathcal{Y}\to\R$ denote a loss function such that $\ell(f(x),y)$ measures the error made when predicting $y$ by $f(x)$. 
For any distribution $Q$ on $\mathcal{Z}$ and any function $g:\mathcal{Z}\to \R$ for which it makes sense, let $Qg=\E_{Z\sim Q}[g(Z)]$ denote the expectation of the function $g$ under the distribution $Q$ and, for any $p\geqslant 1$, let $\|g\|_{L_p(Q)}:=(Q[|g|^p])^{1/p}$ and $\|g\|_{L_p}:=\|g\|_{L_p(P)}$. 
The risk of any $f\in F$ is given by $P\ell_f$, where $\ell_f(x,y):=\ell(f(x),y)$. 
The prediction of $Y$ with minimal risk is given by $f^*(X)$, where $f^*$, called \emph{oracle}, is defined as any function such that
\[
f^*\in \argmin_{f\in F}P\ell_f\enspace.
\]
Hereafter, for simplicity, it is assumed that $f^*$ exists and is uniquely defined.
The oracle is unknown to the statistician that has only access to a dataset $(X_i,Y_i)_{i\in\{1,\ldots,N\}}$ of random variables taking values in $\cX\times\cY$.
The goal is to build a data-driven estimator $\hat{f}$ of $f^*$ that predicts almost as well as $f^*$.
The quality of an estimator $\hat f$ is measured by the error rate $\|\hat{f}-f^{*}\|_{L_2}^2$ and the excess risk $P \cL_{\hat f}$, where, respectively,
\begin{gather} 
\label{error_rate}\|\hat{f}-f^{*}\|_{L_2}^2 = P[(\hat{f}-f)^2]=\mathbb{E} \bigg[  \bigg(\hat{f}(X)-f^{*}(X)\bigg)^2 | (X_i,Y_i)_{i=1}^N   \bigg] \mbox{ and }
\cL_{\hat f} := \ell_{\hat f} - \ell_{f^{*}}\enspace.
\end{gather}

Let $P_N$ denote the empirical measure i.e $P_N(A) = (1/N) \sum_{i=1}^N I(Z_i\in A)$ for all $A\in\cA$. A natural candidate for the estimation of $f^*$ is the Empirical Risk Minimizer (ERM) of \cite{MR0288823}, see also \cite{vapnik1998statistical} for an overview, which is defined by
\begin{equation}
\hat{f}^{ERM} \in  \argmin_{f \in F}  P_N \ell_f  \enspace.
\end{equation}

The choice of $F$ is a central issue: enlarging the space $F$ deteriorates the quality of the oracle estimation but improves its predictive performance.
It is possible to use large classes $F$ without significantly altering the quality estimation if certain structural properties of the oracle $f^*$ are known a priori from the statistician. 
In that case, a widely spread approach is to add to the empirical loss a regularization term promoting this structural property. 
In this paper, we consider this problem when the regularization term is a norm.
Formally, let $E$ be a linear space such that $F\subset E\subset L_2(\mu)$ and let $\|\cdot\|: E \mapsto \bR^+$ denote a norm on $E$. For any $\lambda \geq 0$, the regularized ERM (RERM) is defined by
\begin{equation} \label{eq2}
\hat{f}^{RERM}_{\lambda} \in \argmin_{f \in F}P_N \ell^\lambda_f, \quad\text{where}\quad \ell^\lambda_f(x,y)=\ell_f(x,y) + \lambda \|f\| \enspace.
\end{equation}
In regression, one can mention Thikonov regularization which promotes smoothness \cite{golub1999tikhonov} and $\ell_1$ regularization which promotes sparsity \cite{tibshirani1996regression}. 
Likewise, for matrix reconstruction, the 1-Schatten norm $S_1$ promotes low rank solutions (see \cite{MR2906869,cai2016estimating}). 


In the remaining of the paper, the following notations will be used repeatedly: for any $r>0$, let
\[
rB_{L_2} = \{ f \in L_2(\mu): \|f\|_{L_2} \leqslant r \},\quad rS_{L_2} = \{ f \in L_2(\mu): \|f\|_{L_2} = r \}\enspace. 
\]
Let $rB = \{ f \in E: \|f\| \leq r \}$ and  $rS = \{ f \in E: \|f\| = r \}$. 
For any set $H$ for which it makes sense, let $H + f^{*} = \{ h+f^{*}: h \in H   \}$, $H - f^{*} = \{ h-f^{*}: h \in H   \}$. Let $(e_i)_{i=1}^p$ be the canonical basis of $\mathbb R^p$. Let $c$ denote an absolute constant whose value might change from line to line and let $c(A)$ denote a function depending on the parameters $A$ whose value may also change from line to line.

\section{Regularized ERM with Lipschitz and convex loss functions} \label{sec_ERM}
This section presents and improves results from \cite{pierre2017estimation}. 
A local Bernstein assumption, holding in a neighborhood of the \emph{oracle} $f^{*}$ is introduced in the spirit of \cite{ChiLecLer:2018}. 
This assumption does not imply boundedness of $F$ in $L^2$-norm unlike the global Bernstein condition considered in \cite{pierre2017estimation}.
New rates of convergence are obtained, depending on \textbf{localized} complexity parameters improving the global ones from \cite{pierre2017estimation}. 

\subsection{Main assumptions} 
We start with a set of assumptions sufficient to prove exponential deviation bounds for the error rate and excess risk of RERM for general convex and Lipschitz loss functions and for any regularization norm. In this section, we consider the classical i.i.d. assumption (we will relax this assumption in the next sections in order to consider corrupted databases).
\begin{Assumption} \label{assum_erm1}
	$(X_i,Y_i)_{i=1}^N$ are independent and identically distributed with distribution $P$.
\end{Assumption}
All along the paper, we consider Lipschitz and convex loss functions.
\begin{Assumption}\label{assum:lip}
	There exists $L>0$ such that, for any $y \in \cY$, $\ell(\cdot,y)$ is \textbf{$L$-Lipschitz} i.e for every $f$ and $g$ in $F$, $x \in \cX$ and $y \in \cY$, $|\ell(f(x),y) - \ell(g(x),y) | \leq L |f(x)-g(x)|$ and \textbf{convex} i.e for all $\alpha \in [0,1]$, $\ell(\alpha f(x) + (1-\alpha) g(x),y) \leq \alpha \ell(f(x),y) + (1-\alpha )\ell(g(x),y)$. 
\end{Assumption}
There are many examples of loss functions satisfying Assumption~\ref{assum:lip}. The two examples studied in this work (see Section~\ref{sec:app}) are
\begin{itemize}
	\item the \textbf{logistic loss function} defined for any $u\in\R$ and $y\in\cY=\{-1, 1\}$, by 
	$\ell(u,y) = \log(1+\exp(-yu))$. It satisfies Assumption~\ref{assum:lip} for $L=1$.
	\item  Tte \textbf{Huber loss function} with parameter $\delta >0$ is defined for all $u,y\in\R$, by
	\[
	\ell(u,y) =  
	\begin{cases}
	\frac{1}{2}(y-u)^2&\text{ if }|u-y| \leq \delta\\
	\delta|y-u|-\frac{\delta^2}{2}&\text{ if }|u-y| > \delta
	\end{cases}\enspace.
	\]
	It satisfies  Assumption~\ref{assum:lip} for $L=\delta$.
\end{itemize}
We will also assume that the functions class $F$ is convex.
\begin{Assumption}\label{assum:convex}
	The class $F$ is convex.
\end{Assumption}
In particular, Assumption~\ref{assum:convex} holds in the important case considered in high-dimensional statistics when $F$ is the class of all linear functions indexed by $\R^p$, $F = \{ \inr{t,\cdot}: t \in \mathbb{R}^p  \}$. This example is studied in great details in Section~\ref{sec:app}.

\medskip

RERM performs well when the empirical excess risk $f\in F \to P_N\cL_f$ is uniformly concentrated around the excess risk $f\in F\to P\cL_f$. 
This requires strong concentration properties of the class of random variables $\{\cL_f(X):f\in F\}$, which is implied by concentration properties of $\{(f-f^*)(X):f\in F\}$ thanks to the Lipschitz assumption on the loss function. 
Here, we study RERM under a subgaussian assumption on the design. 
We first recall the definition of a subgaussian class of functions.
\begin{Definition}
	A class $F$ is called $L_0$-subgaussian (with respect to $X$), where $L_0 \geq 1$, when for all $f$ in $F$ and for all $\lambda >1$, $\mathbb E \exp( \lambda |f(X)|/ \|f\|_{L_2} ) \leq \exp(\lambda^2 L_0^2/2)$.
\end{Definition}
\begin{Assumption} 
	\label{ass:sub-gauss} The class $F-f^{*}$ is $L_0$-subgaussian with respect to $X$.
\end{Assumption}

Assumptions \ref{assum_erm1}-\ref{ass:sub-gauss} are also granted in~\cite{pierre2017estimation}.  
In this setup, a natural way to measure the statistical complexity of the problem is via Gaussian mean widths (of some subsets of $F$). We recall the definition of this measure of complexity.
\begin{Definition}\label{def:gauss_mean_width}
	Let $H\subset L_2(\mu)$ and $(G_h)_{h\in H}$ be the canonical centered Gaussian process indexed by $H$, with covariance structure given by $\left(\E (G_{h_1}- G_{h_2})^2\right)^{1/2} = \left(\E(h_1(X)-h_2(X))^2\right)^{1/2}$ for all $h_1,h_2\in H$. The \textbf{Gaussian mean-width} of $H$ is $w(H) = \E \sup_{h\in H} G_h$.
\end{Definition}
Gaussian mean widths of various sets have been computed in \cite{MR3311453},  \cite{bellec2017localized}, or~\cite{gordon2007gaussian} for example. Risk bounds for $\hat{f}^{RERM}_{\lambda}$ are driven by fixed point solutions of a Gaussian mean width of regularization balls $(F-f^*)\cap \rho B$, which measure the local complexity of $F$ around $f^*$. 

\begin{Definition}\label{def:function_r}
	For all $A>0$, the \textbf{complexity function} is a non-decreasing function $r(A,\cdot)$, such that for every $\rho\geq0$, 
	\begin{equation*}
	r(A, \rho) \geq \inf \{ r>0 : 96A L_0Lw\big(F\cap(f^* + \rho B \cap r B_{L_2})\big) \leq r^2\sqrt{N}  \}\enspace.
	\end{equation*}
	Here, $L$ is the Lipschitz constant in Assumption~\ref{assum:lip} and $L_0$ is  the subgaussian constant from Assumption~\ref{ass:sub-gauss}.
\end{Definition} For a given $\rho>0$, parameter $r(A, \rho)$ measures the ''statistical complexity'' of the class $(F-f^*)\cap \rho B$. As one can see in Definition~\ref{def:function_r}, only the complexity locally around $f^*$ matters: it is the Gaussian mean width of  $(F-f^*)\cap \rho B$ intersected with a $L_2$ ball and not the complexity of the entire class $(F-f^*)\cap \rho B$ which appears. The radius of this $L_2$ ball is solution to a fixed point equation as in Definition~\ref{def:function_r}; that is $r(A, \rho)$ is the smallest $r$ such that $\big(F\cap(f^* + \rho B \cap r B_{L_2})\big)$ is of the order of  $r^2\sqrt{N}$.

The last tool and assumption comes from \cite{LM_sparsity}. 
A key observation is that the regularization norm $\|\cdot\|$ promoting some sparsity structure has large subdifferentials at sparse functions (see, for instance, atomic norms in \cite{MR3138795}).   
The subdifferential of $\|\cdot\|$ in $f$ is defined as
\begin{equation} \label{subdiff}
(\partial \|.\|)_f = \{ z^{*} \in E^{*}  \enspace : \enspace \|f+h\| - \|f\| \geq z^{*}(h) \enspace \text{for every  } h \in E  \}\enspace,
\end{equation}
where $E^{*}$ is the dual space of the normed space $(E,\|\cdot\|)$. 
Let
\begin{align*}
 \Gamma_{f^{*}}(\rho) = \bigcup_{f \in f^{*} + \frac{\rho}{20}B} (\partial \|\cdot\|)_f
\end{align*}
be the union of all subdifferentials of the regularization norm $\norm{\cdot}$ of functions $f$ close to the oracle $f^*$. We expect $\Gamma_{f^{*}}(\rho)$ to be a ``large'' subset of the unit dual sphere of $\norm{\cdot}$ when $f^*$ is ``sparse'' -- for the notion of sparsity associated with $\norm{\cdot}$. This intuition is formalized in the following definition from \cite{LM_sparsity}
\begin{Definition}[\cite{LM_sparsity}]\label{def:SE}
	For any $A > 0$ and $\rho>0$, let
	\begin{gather} 
	\notag {H}_{\rho,A} = \set{ f \in F \enspace : \enspace \|f^{*}-f\| = \rho \enspace \text{and} \enspace \|f^{*}-f\|_{L_2} \leq r(A,\rho) }\enspace.
	\end{gather}
	Let 
	\begin{equation} \label{sparisty:eq}
	\Delta(\rho,A) = \inf_{h \in H_{\rho,A} } \sup_{z^{*}  \in \Gamma_{f^{*}}(\rho) }   z^{*}(h-f^{*})\enspace.
	\end{equation}
	A real number $\rho>0$ satisfies the ($A$-)\textbf{sparsity equation} if $\Delta(\rho,A) \geq 4\rho/5$.
\end{Definition}
Any constant in $(0,1)$ could replace $4/5$ in Definition~\ref{def:SE} as can be seen from a close inspection of the proof of Theorem~\ref{thm:main_subgaussian}. 
If the norm $\|\cdot\|$ is ``smooth" in $f$, the subdifferential of $\|\cdot\|$ in $f$  is just the gradient of $\|\cdot\|$ in $f$. 
In that case, $(\partial\norm{\cdot})_f$ is not rich (it is a singleton) and the regularization norm has only a low ``sparsity inducing power'' unless the variety of gradients of $\norm{\cdot}$ at $f$ in the neighborhood $f^{*} + (\rho/20)B$ is rich enough (the latter case can be seen as $\norm{\cdot}$ being ``almost not differentiable'' in $f^*$ since, even though $\norm{\cdot}$ is differentiable in $f^*$, its gradient changes a lot in a small neighborhood of $f^*$). 
However, any norm has a subdifferential in $0$ equal to the entire unit dual ball associated with $\norm{\cdot}$. 
Therefore, when $0$ belongs to $f^{*} + (\rho/20)B$, for example when $\rho \geq 20\|f^*\|$, the sparsity equation is satisfied since, in that case, $\Delta(\rho) = \rho$. 
We can use this fact to obtain ``complexity dependent'' rates of convergence -- i.e. rates depending on $\|f^*\|$. 
In high-dimensional setups, we also look for statistical bounds depending on the sparsity of $f^*$ enforced by $\norm{\cdot}$ (see~\cite{MR3763780,LM_sparsity} for details regarding the difference between ``complexity and sparsity'' dependent bounds). 
Hereafter, we focus on norms $\|\cdot\|$ promoting some sparsity structure and we establish sparsity dependent rates of convergence and sparse oracle inequalities in Section~\ref{sec:app}. 

\medskip

Margin assumptions \cite{MR1765618,MR2051002,MR3526202} such as the Bernstein conditions from \cite{MR2240689} have been widely used in statistics and learning theory to prove fast convergence rates of RERM. 
Here, we use a \textbf{local Bernstein condition} in the spirit of \cite{ChiLecLer:2018}. 
\begin{Assumption} \label{assum:fast_rates}
	There exist constants $A>0$ and $\rho^*$ such that $\rho^*$ satisfies the $A$-sparsity equation and for all $f \in F$ satisfying $\|f-f^*\|_{L_2} = r(A,\rho^*)$ and $\|f-f^*\| \leq \rho^*$, then $\|f-f^*\|_{L_2}^2 \leq A P\cL_f$.
\end{Assumption}
Hereafter, whenever Assumption~\ref{assum:fast_rates} is granted, we assume that the constant $A$ is fixed satisfying this assumption and write $r(\rho)$ instead of $r(A,\rho)$. 
As explained in \cite{ChiLecLer:2018}, the local Bernstein condition holds in examples where $F$ is not bounded in $L_2$-norm. 
It allows to cover the class of all linear functions on $\R^d$ where the global Bernstein condition of \cite{pierre2017estimation} -- $\|f-f^{*}\|_{L_2}^2\leqslant AP\mathcal{L}_f$ for all $f\in F$-- does not hold. 

Finally, the interplay between the complexity parameter, the Bernstein condition and the sparsity equation has been discussed in \cite{LM_sparsity} and \cite{ChiLecLer:2018}.
\begin{Remark} \label{remark:ber}
	From Assumption~\ref{assum:lip} it follows that if the local Bernstein condition is granted as in Assumption~\ref{assum:fast_rates}  that is for all functions $f$ in $F$ such that $\|f-f^*\|_{L_2} = r(A,\rho^*)$ and $\|f-f^*\| \leq \rho^*$ (and if there exists such an $f$) then we necessary have $r(A,\rho^*) \leq AL$. Indeed, if there is an $f$ in $F\cap(f^*+r(A,\rho^*)S_{L_2}\cap \rho^* B)$, it follows from the Lipschitz property of the loss function that 
	\begin{equation*}
r^2(A,\rho^*) = 	\|f-f^*\|_{L_2}^2 \leq AP\cL_f \leq AL \|f-f^*\|_{L_2} = AL r(A,\rho^*)
	\end{equation*}and so $r(A,\rho^*)\leq AL$. The latter condition will be always satisfied as soon as $N$ is large enough. For example, for the LASSO regularization, we recover from the latter restriction, the classical condition  ``$N \gtrsim s\log(ep/s)$'' where $s$ is the oracle's sparsity.\\
	The complexity parameter, the sparsity equation and the local Bernstein are closely related. It is clear that $r(A,\rho)$ is decreasing with $\rho$ for any $A>0$. On this other hand, as we will see in application, to verify the sparsity equation $\rho^*$ cannot be to small. The smallest $\rho^*$ satisfying the sparsity equation leads to the smallest complexity parameter $r(A,\rho^*)$. The next step consists in verifying the local Bersntein assumption for an absolute constant $A>0$. 
\end{Remark}

\subsection{Main theorem for the RERM}

The following theorem gives the main result on the statistical performance of RERM.
\begin{Theorem}\label{thm:main_subgaussian}
Grant Assumptions~\ref{assum_erm1},~\ref{assum:lip},~\ref{assum:convex},~\ref{ass:sub-gauss}. Suppose that Assumption~\ref{assum:fast_rates} holds with $\rho = \rho^{*}$ satisfying the $A$-sparsity equation from Definition~\ref{def:SE}. With this value of $A$, let $r(\cdot):= r(A,\cdot)$ denote the complexity function from Definition~\ref{def:function_r}. Assume that
	\begin{equation}\label{eq:reg_param_choice}
	\frac{10}{21A} \frac{ r^2(\rho^{*})}{\rho^{*}}< \lambda < \frac{2}{3A}\frac{r^2(\rho^{*})}{\rho^{*}}\enspace.
	\end{equation} 
	Then, with probability larger than
	\begin{equation}\label{eq:proba}
	1-2\exp\big(-c(A,L,L_0)r^2(\rho^{*} )N  \big)\enspace,
	\end{equation}
	the following bounds hold
	\begin{align*}
	\|\hat{f}^{RERM}_{\lambda}- f^{*}\| \leq \rho^*,\quad
	\|\hat{f}^{RERM}_{\lambda}- f^{*}\|_{L_2} &\leq r(\rho^*) \mbox{ and }
	P\cL_{\hat{f}^{RERM}_{\lambda}} \leq  \frac{r^2(\rho^*)}A\enspace. 
	\end{align*}
\end{Theorem}
\begin{Remark}
A remarkable feature of Theorem~\ref{thm:main_subgaussian} is that it holds without assumption on $Y$. We do not even need $Y$ to be in $L_1$ since one can always fix some $f_0\in F$ and work with  $\ell_f-\ell_{f_0}$ to define all the object. In that case we have  $|\ell_f-\ell_{f_0}|\leq L|f-f^0|$ and so $(\ell_f-\ell_{f_0})(Z)\in L^1$ when $F\subset L^1(\mu)$ even when $Y\notin L^1$. So we can define $f^*$ such that $f^*\in\argmin_{f\in F}P(\ell_f-\ell_{f_0})$ with no assumption on $Y$. This is an important consequence of the Lipschitz property which has been widely used in robust statistics because it implies robustness to heavy-tailed noise without any strong technical difficulty. 
\end{Remark}

\begin{Remark}
	 Theorem~\ref{thm:main_subgaussian} holds for subgaussian classes of functions $F$. As in \cite{pierre2017estimation}, it is possible to extend this result under boundedness assumptions.
\end{Remark}
Theorem~\ref{thm:main_subgaussian} improves~\cite[Theorem~2.1]{pierre2017estimation} in two directions: First, the complexity function $r(\cdot)$ measures the (Gaussian mean width) complexity of the \textbf{local} set  $(F-f^*)\cap \rho B\cap r B_{L_2}$  and not the global gaussian mean width of $(F-f^*)\cap \rho B $  such as in \cite{pierre2017estimation}. 
Second, Theorem~\ref{thm:main_subgaussian} holds in a setting where $F$ can be unbounded in $L_2$-norm. 
The proof of Theorem~\ref{thm:main_subgaussian} is postponed to Section~\ref{sec:proofs}. 
The proof relies on the convexity of the loss function (and $F$) which allows to use an homogeneity argument as in \cite{ChiLecLer:2018} for Lipshitz and convex loss functions and in \cite{lecue2013learning} for the quadratic loss function, simplifying the peeling step of \cite{pierre2017estimation}. Theorem~\ref{thm:main_subgaussian} is a general result which is applied in various applications in Section~\ref{sec:app}.

\section{Minmax MOM estimators}\label{sec:minmaxMOM}
Even if the results of Section~\ref{sec_ERM} are interesting on their own (because the i.i.d. sub-gaussian framework is one of the most considered setup in Statistics and Learning theory), the setup considered in Section~\ref{sec_ERM} can be restrictive in some applications. 
It does not cover more realistic situations where data are heavy-tailed and/or corrupted. 
In this section, we consider a more general setup beyond the i.i.d. subgaussian setup in order to cover these more realistic frameworks. The results from Section~\ref{sec_ERM} will serve as benchmarks:  we show that similar bounds can be achieved in a more realistic framework by alternative estimators.
These estimators use  the median-of-means principles instead of empirical means.


\subsection{Definition} \label{sec:OUI}



Recall the definition of MOM estimators of univariate means from \cite{MR1688610, MR855970, MR702836}. Let $(B_k)_{k=1,\ldots,K}$ denote a partition of $\{1,\ldots,N\}$ into blocks $B_k$ of equal size $N/K$ (it is implicitly assumed that $K$ divides $N$. An extension to blocks with almost equal size is possible (see \cite{minsker2017distributed}). It is not considered here to simplify the presentation of the results, the extension is thus left to the interested reader). 
For any function $f:\mathcal{X} \times \mathcal{Y} \to \mathbb{R}$ and $k\in \{1,\ldots,K\}$, let $P_{B_k} f = (K/N)\sum_{i \in B_k} f(X_i,Y_i)$ denote the empirical mean on the block $B_k$.
The MOM estimator based on this partition is the empirical median of the latter empirical means:
\begin{equation}\label{MOM}
	\MOM{K}{f} = \text{Med}(P_{B_1} f,\cdots,P_{B_K} f) \enspace.
\end{equation} 
The estimator $\MOM{K}{f}$ of $Pf$ achieves subgaussian deviation tails if $(f(X_i,Y_i))_{i =1}^N$ have $2$ moments, see \cite{devroye2016sub}.
The number of blocks $K$ is a tuning parameter of the procedure. The larger $K$, the more outliers are allowed. When $K=1$, $\MOM{K}{f}$ is the empirical mean, when $K=N$, it is the empirical median. 


Building on ideas introduced in \cite{MR2906886,MR3595933}, \cite{lecue2017robust} proposed the following strategy to use MOM estimators in learning problems. Since the \emph{oracle} $f^*$ is also solution of the following minmax problem
\begin{align*}
f^{*} = \argmin_{f \in F} P\ell_f = \argmin_{f \in F}  \sup_{g \in F} P(\ell_f-\ell_g)\enspace,
\end{align*}
minmax MOM estimators are obtained by plugging MOM estimators of the unknown expectations $P(\ell_f-\ell_g)$ in this minmax formulation. 
Applying this principle to regularized procedures yields the following ``minmax MOM version'' of RERM that we study in this paper:
\begin{equation}\label{def:MOM}
\hat{f}_{K,\lambda} \in \argmin_{f \in F} \sup_{g \in F} \MOM{K}{\ell_f-\ell_g} + \lambda\big(  \|f\|-\|g\| \big)\enspace.
\end{equation}

The linearity of the empirical process $P_N$ is important to use localization techniques in the analysis of RERM to derive fast rates of convergence for these estimators improving upon the slow rates of \cite{vapnik1998statistical}, see \cite{MR2051002,MR2829871} for example. 
The minmax reformulation comes from \cite{MR2906886}, it allows to overcome the lack of linearity of robust mean estimators and obtain fast rates of convergence for robust estimators based on nonlinear estimators of univariate expectations. 

\subsection{Assumptions and main results}
To highlight robustness properties of minmax MOM estimators with respect to outliers in the dataset, their analysis is performed in the following framework. Let $\mathcal{I} \cup \mathcal{O}$ denote a partition of $\{1,\cdots, N\}$ that is unknown to the statistician. 
Data $(X_i,Y_i)_{i \in \mathcal{O}}$ are considered as outliers. \textbf{No assumption} on the distribution of these data is made, they can be dependent or adversarial. Data $(X_i,Y_i)_{i \in \mathcal{I}}$ bring information on $f^*$ and are called informative or inliers. Assumptions are made uniquely on these informative data (and not on the outliers). They have to induce the same $L_2$ geometries on $F$ and the same excess risks.
\begin{Assumption}\label{assum:moments}
	$(X_i,Y_i)_{i \in \cI}$ are independent and for all $ i  \in \mathcal{I}: P_i(f-f^{*})^2 = P(f-f^{*})^2 $ and $ P_i\mathcal{L}_f  = P\mathcal{L}_f $  \enspace.
\end{Assumption} 
Assumption \ref{assum:moments} holds in the i.i.d case, it also covers situations where informative data $(X_i,Y_i)_{i \in I}$ may have different distributions. It implies in particular that $f^*$ is also the oracle in $F$ w.r.t. all the distributions $P_i$ for $i\in \cI$.\\
Several quantities introduced to study RERM have to be modified to state the results for minmax MOM estimators.
First, the complexity function is no longer based on Gaussian mean width, it is now defined as a fixed point of local Rademacher complexities \cite{koltchinskii2011empirical,MR2329442,MR2040404,bartlett2005local}.  Let $(\sigma_i)_{i \in \cI}$ denote i.i.d. Rademacher random variables (i.e. uniformly distributed on $\{-1, 1\}$), independent from $(X_i,Y_i)_{i \in \cI}$. 
The \textbf{complexity function $\rho\to r_2(\gamma,\rho)$} is a non-decreasing function such that for all $\rho>0$
\begin{equation} \label{comp:rad}
r_2(\gamma,\rho) \geq \inf \bigg\{ r>0: \forall  J \subset \mathcal{I} \mbox{  s.t  }|J| \geqslant N/2, \quad \mathbb{E}\cro{\sup_{f \in (F-f^*)\cap \rho B \cap  rB_{L_2}} \bigg |{\sum_{i \in J} \sigma_i f(X_i)} \bigg |} \leq \gamma r^2 |J|  \bigg\} \enspace.
\end{equation}
As in Theorem~\ref{thm:main_subgaussian}, parameter $r_2(\gamma,\rho)$ measures the statistical complexity of the sub-model $F\cap (f^* + \rho B)$ locally in a $L_2$-neighborhood of $f^*$. 
It only involves the distribution of informative data and does not depend on the distribution of the outputs $(Y_i)_{i \in \cI}$.
%
%
The local Bernstein condition, Assumption~\ref{assum:fast_rates}, as well as the sparsity equation have now to be extended to this new definition of complexity. 
We start with the sparsity equation.
\begin{Definition}\label{def:SEMOM}
For any $A > 0$ and $\rho>0$, let
\begin{gather} 
\label{def_CK}	C_{K,r} (\rho,A)=  \max\left( r_2^2(\gamma,\rho), c(A,L)\frac{K}{N}\right)
\end{gather}and $ \tilde{H}_{\rho,A} = \set{ f \in F \enspace : \enspace \|f^{*}-f\| = \rho \enspace \text{and} \enspace \|f^{*}-f\|_{L_2} \leq \sqrt{C_{K,r}(\rho,A)} }$. Let 
\begin{equation} \label{sparisty:eq}
\tilde{\Delta}(\rho,A) = \inf_{h \in \tilde{H}_{\rho,A} } \sup_{z^{*}  \in \Gamma_{f^{*}}(\rho) }   z^{*}(h-f^{*})\enspace.
\end{equation}
A real number $\rho>0$ satisfies the $A$-\textbf{sparsity equation} if $\tilde{\Delta}(\rho,A) \geq 4\rho/5$.
\end{Definition}
The value of $c(A,L)$ in Definition~\ref{def:SEMOM} is made explicit in Section~\ref{proof_mom}. To simplify the presentation we write $c(A,L)$ as it is an absolute constant depending only on $A$ and $L$. With this definition in mind, one can extend the local Bernstein assumption.
\begin{Assumption}\label{assum:fast_rates_MOM} 
There exist a constant $A > 0$ and $\rho^*$ such that $\rho^*$ satisfies the $A$-sparsity equation from Definition~\ref{def:SEMOM} and, for all $f\in F$ such that $\|f-f^{*}\|_{L_2}^2 = C_{K,r} (2\rho^{*},A)$ and $\|f-f^*\| \leq 2\rho^*$, $\|f-f^{*}\|_{L_2}^2\leqslant AP\mathcal{L}_f$. 
\end{Assumption}
As in Assumption~\ref{assum:fast_rates}, the link between $\|f-f^*\|_{L_2}^2$ and the excess risk $P\mathcal{L}_f$ in Assumption~\ref{assum:fast_rates_MOM} is only granted in a $L_2(\mu)$-sphere around the oracle $f^*$ whose radius is proportional to the rate of convergence of the estimators (see Theorems~\ref{thm:main_subgaussian} and~\ref{main:theo}).
The local Bernstein assumption is somehow ``minimal'' since it is only granted on the smallest set of the form $F\cap (f^*+2\rho^* B \cap r_2(\gamma, 2\rho^*)B_{L_2})$ centered in $f^*$ that can be proved to contain $\hat f_{K,\lambda}$ (when $K$ is such that $\sqrt{C_{K,r}(2\rho^*,A)} = r_2(\gamma,2\rho^*)$). 
\begin{Remark}
	As in Remark~\ref{remark:ber} we necessary have $\sqrt{C_{K,r}(2\rho^*,A)} \leq AL$ under Assumption~\ref{assum:fast_rates_MOM} and the Lipschitz assumption from Assumption~\ref{assum:lip}. This is also this condition which requires a minimal number of observations to hold out of which we recover the classical conditions such as $N\gtrsim s \log(ep/s)$ when one wants to reconstruct a $s$-sparse vector.
\end{Remark}
We are now in position to state our main result on the statistical performances of the regularized minmax MOM estimator.
\begin{Theorem}\label{main:theo}
		Grant Assumptions~\ref{assum:lip},~\ref{assum:convex},~\ref{assum:moments} and~\ref{assum:fast_rates_MOM} for  $\rho^*$ satisfying the $A$-sparsity equation from Definition~\ref{def:SEMOM}. Let $K \geq  7|\mathcal{O}|/3 $, $\gamma = 1/(6528L)$, and define \
	\begin{equation*}
	\lambda =  \frac{5}{17A} \frac{C_{K,r}(2\rho^{*},A)}{\rho^*}\enspace.
	\end{equation*}
	Then, with probability larger than $1- 2\exp(-cK)$, the minmax MOM estimator $\hat f_{K,\lambda}$ defined in \eqref{def:MOM} satisfies
		\begin{align*}
		& \|\hat{f}_{K,\lambda}-f^{*}\| \leq 2\rho^{*}, \quad  \|\hat{f}_{K,\lambda}-f^{*}\|_{L_2}^2 \leq C_{K,r}(2\rho^{*},A) \quad \mbox{and} \quad  P\mathcal{L}_{\hat{f}_{K,\lambda}} \leq  \frac{1}{A} C_{K,r}(2\rho^{*},A)\enspace.
		\end{align*}
\end{Theorem}
Suppose that $K = c(A,L) r_2^2(\gamma,2\rho^{*})N $, which is possible as long as $|\cO|\leq c(A,L)N r_2^2(\gamma,2\rho^{*})$. 
The $L_2$-estimation bound obtained in Theorem~\ref{main:theo} is then $r_2^2(\gamma,2\rho^{*})$ and the probability that this bound holds is $1-\exp(- c(A,L)N r_2^2(\gamma,2\rho^{*}))$.  Up to absolute constants, regularized minmax MOM estimators achieve the same bounds as RERM with the same probability when the inlier data satisfy the subgaussian assumption as in the framework of Theorem~\ref{thm:main_subgaussian}.
Indeed, in that case, a straightforward chaining argument shows that the Rademacher complexity from \eqref{comp:rad} is upper bounded by the Gaussian mean width. The difference with Theorem~\ref{thm:main_subgaussian} is that the estimator depends on $K$. On the other hand, the results from Theorem~\ref{main:theo} hold in a setting beyond the subgaussian assumption on $F$ and the data may not be identically distributed and may have been corrupted by outliers. In Section~\ref{sec_mom_lasso}, we consider an example where rate optimal bounds can be derived from this general result under weak moment assumptions while still achieving the same rate as in the sub-gaussian framework. It is also possible to adapt in a data-driven way to the best $K$ and $\lambda$ by using a Lepski's adaptation method such as in \cite{devroye2016sub,lecue2017learning,lecue2017robust,ChiLecLer:2018,chinot2019robust}. This step is now well understood, it is not reproduced here. Theorem~\ref{main:theo} is general result in the sense that it allows to handle many applications where a convex and Lipschitz loss function and a regularization norm are used (some examples are presented in Section~\ref{sec:app}).


\section{Relaxing the Bernstein condition} \label{sec_wtb}
In this section, we study minmax MOM estimators when the Bernstein assumption~\ref{assum:fast_rates_MOM} is relaxed. 
The price to pay for this relaxation is that, on one hand, the $L_2$-risk is not controlled and on the other hand an extra variance term appears in the excess risk $P\cL_{\hat f_K^{\lambda}}$. Nevertheless, under a slightly stronger local Bernstein's condition, the extra variance term can be controled and the bounds from Theorem~\ref{main:theo} can be recovered. We consider the following assumption which is weaker than Assumption~\ref{assum:moments} since it does not require that the distribution of the $X_i$'s, for $i\in\cI$ induce the same $L_2$ structure as the one of $L_2(\mu)$.
\begin{Assumption}\label{assum:moments2}
	$(X_i,Y_i)_{i \in \mathcal{I}}$ are independent and for all $i\in\cI$, $(X_i,Y_i)$ has distribution $P_i$, $X_i$ has distribution $\mu_i$. We assume that, for any $i\in \cI$, $F\subset L_1(\mu_i)$  and $ P_i\mathcal{L}_f  = P\mathcal{L}_f$ for all $f\in F$.
\end{Assumption}
Since the local Bernstein Assumption~\ref{assum:fast_rates_MOM} does not hold, the localization argument has to be modified. Instead of using the $L_2$-norm to define neighborhoods of $f^*$ as in the previous section, we use the excess loss $f\in F \to P\cL_f$ as proximity function defining the neighborhoods. The new fixed point is defined for all $\gamma,\rho >0$ and $K \in \{1,\cdots,N\}$:

\begin{equation} \label{eq:comp_par_local_excess_loss}
\bar{r}(\gamma,\rho)  = \inf \bigg\{ r > 0: \max \bigg(\frac{E(r,\rho)}{\gamma}, \sqrt{c}V_K(r,\rho) \bigg)\leq r^2 \bigg\},\quad \mbox{where}
\end{equation}
\begin{align*}
& E(r,\rho) = \sup_{J \subset \mathcal{I}:|J|\geq N/2}\mathbb{E}\sup_{f \in F: P\cL_f \leq r^2, \; \|f-f^*\| \leq \rho} \bigg | \frac{1}{ |J|}\sum_{i \in J} \sigma_i(f-f^{*})(X_i)   \bigg |\enspace,\\
& V_K(r,\rho) =\max_{i\in\cI}\sup_{f \in F: P\cL_f \leq r^2, \; \|f-f^*\| \leq \rho}\left(\sqrt{\mathbb{V}ar_{P_i}(\cL_f)}\right) \sqrt{\frac{K}{N}}\enspace,
\end{align*}
and $(\sigma_i)_{i \in \cI}$ are i.i.d. Rademacher random variables independent from $(X_i,Y_i)_{i \in \cI}$. The value of $c$ in Equation~\eqref{eq:comp_par_local_excess_loss} can be found in Section~\ref{sec:MOM_wtb}. 
The main differences between $ r_2(\gamma,\rho)$ in~\eqref{comp:rad}  and $\bar r(\gamma,\rho)$ in \eqref{eq:comp_par_local_excess_loss}  are the extra variance $V_K$ term and the $L_2$ localization which is replaced by an "excess of risk" localization. 
Under the local Bernstein Assumption~\ref{ass:bernstein_kappa} below, this extra variance term $V_K(r,\rho)$ becomes negligible in front of the complexity term $E(r,\rho)$. In that case, the fixed point $\bar r(\gamma, r)$ matches the $r_2(\gamma, \rho)$ used in Theorem~\ref{main:theo}. 
As in Section~\ref{sec:minmaxMOM}, the sparsity equation has to be modified according to this new definition of fixed point. 
\begin{Definition}\label{def:SEMOM2}
	For any $\rho>0$, let
	\begin{gather} 
	\bar{H}_{\rho} = \set{ f \in F \enspace : \enspace \|f^{*}-f\| = \rho \enspace \text{and} \enspace P\cL_f \leq \bar{r}^2(\gamma,\rho) }\enspace.
	\end{gather}
	Let 
	\begin{equation} \label{sparisty:eq2}
	\bar{\Delta}(\rho) = \inf_{h \in \bar{H}_{\rho} } \sup_{z^{*}  \in \Gamma_{f^{*}}(\rho) }   z^{*}(h-f^{*})\enspace.
	\end{equation}
	A real number $\rho>0$ satisfies the \textbf{sparsity equation} if $\bar{\Delta}(\rho) \geq 4\rho/5$.
\end{Definition}
We are now in position to state the main result of this section.

\begin{Theorem}\label{theo:main_without_bernstein_cond}
	Grant Assumptions~\ref{assum:lip},~\ref{assum:convex},~\ref{assum:moments2} and assume that $|\cO|\leq 3N/7$. Let $\rho^*$ satisfying the sparsity equation from Definition~\ref{def:SEMOM2}. Let $\gamma=1/(3840L)$ and  $K \in \big[  7|\cO|/3 ,  N \big]$. Define
	\begin{equation*}
	\lambda = \frac{11}{40}\frac{\bar{r}^2(\gamma,2\rho^*)}{\rho^*}
	\end{equation*}
	The minmax MOM estimator $\hat{f}_{K,\lambda}$ defined in \eqref{def:MOM} satisfies, with probability at least $1- 2\exp(-cK)$, 
	\begin{equation*}
	P\cL_{\hat f_{K,\lambda}}\leq \bar{r}^2(\gamma,2\rho^*)\quad \mbox{and} \quad \|\hat f_{K,\lambda}-f^*\| \leq 2\rho^*  \enspace.
	\end{equation*}
\end{Theorem}
In Theorem~\ref{theo:main_without_bernstein_cond}, the only stochastic assumption is Assumption~\ref{assum:moments2} which says that the inliers data are independent and define the same excess risk as $(X,Y)$ over $F$. In particular, Theorem~\ref{theo:main_without_bernstein_cond} does not assume anything on the outliers $(X_i,Y_i)_{i\in\cO}$ nor on the outputs of the inliers $(Y_i)_{i\in\cI}$ like in the previous section but it also does not require any other assumption than the existence of all the considered objects. It follows from Theorem~\ref{theo:main_without_bernstein_cond} that all the difficulty of the problem is now contained in the computation of the local Rademacher complexities $E(r, \rho)$. 

To conclude the section, let us show that Theorem~\ref{main:theo} can be recovered from Theorem~\ref{theo:main_without_bernstein_cond} under the following  local Bernstein assumption which is slightly stronger than the one assumed in Theorem~\ref{theo:main_without_bernstein_cond}.

\begin{Assumption}\label{ass:bernstein_kappa} There exist a constant $\bar A > 0$ and $\rho^*$ satisfying the sparsity equation from Definition~\ref{def:SEMOM2} such that, for all $f\in F$, if $P\mathcal L _f \leqslant \bar{C}_{K,r}(\rho^*,\bar A)$ and $\|f-f^*\| \leq 2\rho^*$, then $\|f-f^{*}\|_{L_2 }^{2} \leqslant \bar AP\mathcal{L}_f $, where
	\begin{equation}\label{eq:CKr}
	\bar{C}_{K,r}(\rho, A) =  \max\left( \frac{r_2^{2}(\gamma/  A,2\rho)}{\sqrt{ A}}, c(A,L) \frac{K}{N} \right) \quad \mbox{and} \quad \gamma =1/(3840L)\enspace.
	\end{equation}
\end{Assumption}
Up to constants, $\bar{C}_{K,r}$ is equivalent to $C_{K,r}$ given                in Definition~\ref{def:SEMOM}. Assumption~\ref{ass:bernstein_kappa} is a condition on all functions $f \in F$ such that $P\cL_f \leq \bar C_{K,r}(\rho^*,\bar A)$ which is a  slightly stronger condition than being in the $L_2$-sphere as in Assumption~\ref{assum:fast_rates_MOM}.

\begin{Theorem}\label{theo:bernstein_kappa}
	Grant Assumptions~\ref{assum:lip},~\ref{assum:convex},~\ref{assum:moments} and assume that $|\cO|\leq 3N/7$. Assume that the local Bernstein condition Assumption~\ref{ass:bernstein_kappa} holds with $\rho^*$ satisfying the $\bar A$-sparsity equation from Definition~\ref{def:SEMOM2}. Let $\gamma=1/(3840L)$ and  $K \in \big[  7|\cO|/3 ,  N \big]$. Define
	\begin{equation*}
	\lambda = \frac{11}{40}\frac{\bar{r}^2(\gamma,2\rho^*)}{\rho^*}\enspace.
	\end{equation*}
	The minmax MOM estimator $\hat{f}_{K,\lambda}$ defined in \eqref{def:MOM} satisfies, with probability at least $1- 2\exp(-cK)$, 
	\begin{gather*}
	||\hat f_{K,\lambda}-f^*||_{L_2}^2\leq \bar{C}_{K,r}(\rho^*,\bar A), \quad   P\cL_{\hat f_{K,\lambda}} \leq \bar{C}_{K,r}(\rho^*,\bar A) \quad \mbox{and} \quad  \|\hat{f}_{K,\lambda}- f^*\| \leq 2\rho^*\enspace.
	\end{gather*}
\end{Theorem}
Theorem~\ref{theo:bernstein_kappa} is proved in Section~\ref{ProofBernsteinKappa}.

\begin{Remark}

 Under Assumption~\ref{ass:bernstein_kappa} and a slight modification in the constants, $\rho^*$ satisfies the sparsity equation of Definition~\ref{def:SEMOM2} if it verifies the sparsity equation of Definition~\ref{def:SEMOM}. 

\end{Remark}

\section{Applications} \label{sec:app}

This section presents some applications of Theorem~\ref{main:theo} to derive statistical properties of regularized minmax MOM estimators for various choices of loss functions and regularization norm. To check the assumptions of the Theorem~\ref{main:theo}, the following routine is applied:
\begin{itemize}
	\item[1.] Check Assumptions~\ref{assum:lip},~\ref{assum:convex},~\ref{assum:moments}.
	\item[2.] Compute the local rademacher complexity $r_2(\gamma,\rho)$.
	\item[3.] Solve the sparsity equation from Definition~\ref{def:SEMOM}: find $\rho^{*}$ such that $\Delta(\rho^{*},A) \geq 4\rho^{*}/5$.
	\item[4.] Check the local Bernstein condition from Assumption~\ref{assum:fast_rates_MOM}. 
\end{itemize}

 In this section, we focus on high dimensional statistical problems with sparsity inducing regularization norms \cite{MR3025128} such as the $\ell_1$ norm \cite{tibshirani1996regression}, the SLOPE norm \cite{MR3418717}, the group LASSO norm \cite{simon2013sparse}, the Total Variation norm~\cite{osher2005iterative}. We consider the class of linear functions $F = \{ \inr{t,\cdot}: t \in \bR^p  \} $ indexed by $\bR^p$.
 We denote by $t^{*}\in \R^p$ the vector such that $f^{*}(\cdot) = \inr{t^{*}, \cdot}$. 
 We consider the logistic loss function for the LASSO and the SLOPE, with data $(X_i,Y_i)_{i=1}^N$ taking values in $\bR^p \times \{-1, 1\}$ and the Huber loss function for the Group LASSO and the Total Variation with data $(X_i,Y_i)_{i=1}^N$ taking values in $\bR^p \times \bR$. In particular, the results of this section extend results on the logistic LASSO and logistic SLOPE from \cite{pierre2017estimation} and present new results for the Group Lasso and the Total Variation. 

\subsection{Preliminary tools and results}
In this section, we recall some tools to check the Local Bernstein condition, compute the local Rademacher complexity and verify the sparsity equation.
\subsubsection{Local Bernstein conditions for the logistic and Huber loss functions} 
\label{ssub:the_local_bernstein_conditions_for_the_logistic_huber_and_quantile_loss_functions}

In this section, we recall some results from \cite{ChiLecLer:2018} on the local Bernstein condition for the logistic and Huber loss functions. 

For the logistic loss function (i.e. $\ell_f:(x,y)\in\bR^p\times\{\pm1\}\to \log(1+\exp(-yf(x)))$), we first introduce the following assumption. Note that we do not use the full strength of the approach since we check the inequality $\|f-f^*\|_{L_2}^2\leqslant A  P\mathcal{L}_f$ for all $f\in F\cap (f^*+rB_{L_2})$ instead of just all functions in $ F\cap (f^*+rS_{L_2}\cap \rho B)$.

\begin{Assumption}\label{ass:bernstein_logistic_loss}Let $\varepsilon > 0$, there are constants $C'$ and $c_0>0$ such that 
\begin{itemize}
	\item[a)]  for all $f$ in $F$, $\|f-f^*\|_{L_{2+\varepsilon}} \leq C' \|f-f^*\|_{L_2}$
	\item[b)]  $\mathbb P (|f^*(X)| \leq c_0 ) \geq 1-1/(2C')^{(4+ 2\varepsilon)/\varepsilon}$
\end{itemize}
\end{Assumption}
 Under Assumption~\ref{ass:bernstein_logistic_loss}, we check the Bernstein condition on the entire $L_2$-ball of radius $r$ around $f^*$. 


\begin{Proposition}[\cite{ChiLecLer:2018}, \textcolor{red}{Theorem 9}]\label{prop:bernstein_logistic_loss}Grant Assumption~\ref{ass:bernstein_logistic_loss}. Let $r>0$. The local Bernstein condition holds for the logistic loss function: for all $f\in F$ if $\|f-f^*\|_{L_2}\leqslant r$ then $\|f-f^*\|_{L_2}^2\leqslant A  P\mathcal{L}_f$ for
	\begin{equation*}
	A = \frac{\exp \big( -  c_0 -   r(2C')^{(2+\varepsilon)/\varepsilon}   )\big)}{ 2\bigg(1 + \exp \big(    c_0 +   r(2C')^{(2+\varepsilon)/\varepsilon}  \big) \bigg)^2} \enspace.
	\end{equation*}
\end{Proposition}
When $r$ is such that  $r(2C')^{(2+\varepsilon)/\varepsilon} \leq c_0/2$ then $A$ is an absolute constant. In the sequel, $r$ plays the role of the rate of convergence of the estimator and thus the price to pay for assuming this latter condition is on the number of observations: we will for instance recover the classical assumption $N\gtrsim s \log(ep/s)$ for the reconstruction of a $s$-sparse vector from this assumption (see also Remark~\ref{remark:ber} where this type of assumption on $r$ is also needed).


For the Huber loss function with parameter $\delta>0$ (i.e. $\ell_f(x,y)=\rho_\delta(y-f(x))$ where $\rho_\delta(t) = t^2/2$ if $|t|\leq \delta$ and $\rho_\delta(t)=\delta|t|-\delta^2/2$ if $|t|\geq \delta$), we use the following result also borrowed from \cite{ChiLecLer:2018}. 
Let us introduce the following assumption.
\begin{Assumption}\label{ass:bernstein_huber_loss}Let $\varepsilon >0$ and let $F_{Y|X=x}$ be the conditional cumulative function of $Y$ given $X=x$.
	\begin{itemize}
		\item[a)] There exists a constant $C'$ such that, for all $f$ in $F$, $\|f-f^*\|_{L_{2+\varepsilon}} \leq C' \|f-f^*\|_{L_2}$.
		\item[b)] Let $C'$ be the constant defined in a). There exist $r>0$ and $\alpha>0$ such that, for all $x\in\cX$ and all $z\in \mathbb{R}$ satisfying $ |z-f^{*}(x) | \leq r (\sqrt 2 C')^{(2+ \varepsilon)/\varepsilon} $, $F_{Y|X=x}(z+\delta) - F_{Y|X=x}(z- \delta)\geqslant \alpha$.
		
	\end{itemize}
\end{Assumption}
We will use this result when $r$ is the rate of convergence of the estimator. Note that if $r$ is larger than the order of a constant the point b) can be verified only if $\delta$, the Lipschitz constant, is large enough and $\alpha$ is small enough. To avoid this situation we assume that $r (\sqrt 2 C')^{(2+ \varepsilon)/\varepsilon}\leq c$ where $c$ is some absolute constant. In that case, $\delta$ and $\alpha$ can be considered like constants. Again the price we pay for that assumption will be on the number of observations such as the classical one  $N\gtrsim s \log(ep/s)$ for the reconstruction of a $s$-sparse vector. The point b) in Assumption~\ref{ass:bernstein_huber_loss} simply means that the noise puts enough mass locally around 0. It is a very weak condition that holds for heavy-tailed noise (see~\cite{chinot2020erm}). For example, let us assume that $Y= f^*(X) + \xi$, where $\xi$ is a standard Cauchy random variable independent to $X$. Then  
\begin{align*}
F_{Y|X=x}(z+\delta) - F_{Y|X=x}(z- \delta) & = \mathbb P \big( z-\delta \leq Y \leq z+ \delta | X=x\big) \\
& =  \mathbb P \big( z - f^*(x) -\delta \leq \xi \leq z- f^*(x)+ \delta  | X=x\big) \\
& \geq  F_{\xi}(\delta - r (\sqrt 2 C')^{(2+ \varepsilon)/\varepsilon} )  - F_{\xi}( r (\sqrt 2 C')^{(2+ \varepsilon)/\varepsilon} -\delta )  \enspace,
\end{align*}
for every $z\in \mathbb{R}$ satisfying $ |z-f^{*}(x) | \leq r (\sqrt 2 C')^{(2+ \varepsilon)/\varepsilon} $, where $F_{\xi}$ denotes the cumulative distribution function of $\xi$ that is $F_{\xi}(t) = 1/2 + \arctan(t)/\pi$, for $t \in \mathbb R$. It follows that 
$$
F_{Y|X=x}(z+\delta) - F_{Y|X=x}(z- \delta)  \geq \frac{2}{\pi} \arctan(\delta - r (\sqrt 2 C')^{(2+ \varepsilon)/\varepsilon} ) \enspace.
$$
As a consequence, the point b) in Assumption~\ref{ass:bernstein_huber_loss} is verified if 
$$
\frac{2}{\pi} \arctan(\delta - r (\sqrt 2 C')^{(2+ \varepsilon)/\varepsilon} ) \geq \alpha \enspace.
$$The latter condition will hold when $r$ is smaller than some constant and $\delta$ is larger than an other one. We will meet these conditions later on as well.
\begin{Proposition}[\cite{ChiLecLer:2018}, \textcolor{red}{Theorem 7}]\label{prop:bernstein_huber_loss}Grant Assumption~\ref{ass:bernstein_huber_loss} for $ r>0 $. The Huber loss function with parameter $\delta>0$ satisfies the Bernstein condition: for all $f\in F$, if $\norm{f-f^*}_{L_2}\leq r$ then $(4/\alpha) P\cL_f \geq  \norm{f-f^*}_{L_2}^2$.
\end{Proposition}
Let us come back to our example of Cauchy noise. The local Bernstein condition is verified with
$$
A = \frac{4}{\frac{2}{\pi} \arctan(\delta - r (\sqrt 2 C')^{(2+ \varepsilon)/\varepsilon} )}\enspace,
$$
which is of the order of a constant when $\delta$ is a also of the order of a constant and $r$ smaller than another absolute constant (which will be the case for $N$ large enough). This example reveals that the Huber loss function allows to deal with very heavy-tailed noise. 

\subsubsection{Local Rademacher complexities and Gaussian mean widths} 
\label{ssub:local_rademacher_complexities_and_gaussian_mean_width}
The computation of $r_2(\gamma,\rho)$ may be involved, but can sometimes be reduced to the computation of Gaussian mean widths. A typical result in that direction is the one from \cite{mendelson2017multiplier}.
The results of \cite{mendelson2017multiplier} are based on the concepts of unconditional norm and isotropic random vectors. 
\begin{definition} \label{K:uncon}
	For a given vector $x = (x_i)_{i=1}^p$, let $(x_i^{*})_{i=1}^p$ be the non-increasing rearrangement of $(|x_i|)_{i=1}^p$. The norm $\|\cdot\|$ in $\bR^p$ is said $\kappa$-unconditional with respect to the canonical basis $(e_i)_{i=1}^p$ if, for every $x$ in $\bR^p$ and every permutation $\pi$ of $\{ 1,\cdots,p \}$,
	\begin{equation*}
	\norm{\sum_{i=1}^p x_i e_i} \leq \kappa \norm{\sum_{i=1}^p x_{\pi(i)} e_i }\enspace,
	\end{equation*}
	and, for any $y\in \bR^p$ such that, for all $1 \leq i \leq p$, $x_i^{*} \leq y_i^{*}$, then
	\begin{equation*}
	\norm{\sum_{i=1}^p x_i e_i} \leq \kappa \norm{\sum_{i=1}^p y_i e_i }\enspace.
	\end{equation*}
\end{definition}
Typical examples of  $\kappa$-unconditional norms can be found in \cite{mendelson2017multiplier}. In the following we use the fact that the dual norms of the $\ell_1$ and SLOPE norms are $1$-unconditional.
\begin{definition}
	A random vector $X$ in $\bR^p$ is isotropic if $\mathbb E  [\inr{t,X}^2] = \|t\|_2^2$, for all $t\in \R^p$, where $\|\cdot\|_2$ is the Euclidean norm in $\bR^p$. 
\end{definition}
Recall the main result of \cite{mendelson2017multiplier}.
\begin{Theorem} \label{theo:shahar}\cite[Theorem 1.6]{mendelson2017multiplier}
	Let $C_0 $, $\kappa$ and $M$ be real numbers. Let $V \subset \bR^p$ be such that   $\sup_{v \in V} |\inr{v,\cdot}|$ is $\kappa$-unconditional with respect to $(e_i)_{i=1}^p$. Assume that $X\in \bR^p$ is isotropic and satisfies, for all $1\leq j \leq p$  and  $1\leq q \leq C_0 \log(p)$,
		\begin{equation} \label{cond:shahar}
		 \norm{\inr{X, e_j}}_{L_q} \leq M \sqrt{q}\enspace.
		\end{equation}
Let $X_1, \ldots,X_N$ denote independent copies of $X$, then there exists a constant $c_2$ depending only on $C_0$ and $M$ such that
\begin{equation*}
\mathbb{E} \cro{\sup_{v \in V} \sum_{i=1}^N \sigma_i \inr{X_i,v} } \leq c_2 \kappa \sqrt{N} w(V)
\end{equation*}where $w(V)$ is the Gaussian mean width of $V$.
\end{Theorem}
Recall that a real valued random variable $Z$ is $L_0$-subgaussian if and only if for all $q\geq1, \norm{Z}_{L_q}\leq c_0 L_0 \sqrt{q}$, for some absolute constant $c_0$, see Theorem~1.1.5 in \cite{chafai2012interactions}.
Hence, Theorem~\ref{theo:shahar} shows that $C_0\log(p)$ ``subgaussian" moments for the coordinates of the design $X$ are enough to upper bound the Rademacher complexity by the Gaussian mean width. Such a result is useful to show that minmax MOM estimators can achieve the same rate as the ERM (in the subgaussian framework) even when the data are heavy-tailed data. 

\subsubsection{Sub-differential of a norm} 
\label{ssub:sub_differential_of_a_norm}
To solve the sparsity equation -- find $\rho^*$ such that $\tilde{\Delta}(\rho^*, A)\geq 4\rho^*/5$ -- from Definition~\ref{def:SEMOM}, we use the following classical result on the sub-differential of a norm: if $\norm{\cdot}$ is a norm on $\bR^p$, then, for all $t\in\bR^p$, we have
\begin{equation}\label{eq:sub_diff_norm}
(\partial \norm{\cdot})_t=\left\{
\begin{array}{cc}
\{z^*\in S^*:\inr{z^*, t}=\norm{t}\} & \mbox{ if } t\neq0\\
B^* & \mbox{ if } t=0
\end{array}
\right.\enspace.
\end{equation}
Here, $B^*$ is the unit ball of the dual norm associated with $\norm{\cdot}$, i.e. $t\in\bR^p\to\norm{t}^* = \sup_{\norm{v}\leq1}\inr{v,t}$ and $S^*$ is its unit sphere. 
In other words, when $t\neq0$, the sub-differential of $\norm{\cdot}$ in $t$ is the set of all vectors $z^*$ in the unit dual sphere $S^*$ which are norming for $t$ (i.e. $z^*$ is such that $\inr{z^*,t}=\norm{t}$). In particular, when $t\neq0$, $(\partial \norm{\cdot})_t$ is a subset of the dual sphere $S^*$. 

In the following, understanding the sub-differentials of the regularization norm is a key point for solving the sparsity equation. 
If one is only interested in proving ``complexity'' dependent bounds -- which are bounds depending on $\norm{t^*}$ and not on the sparsity of $t^*$ -- then one can simply take $\rho^*=20 \norm{t^*}$. 
Actually, in this case, $0\in\Gamma_{t^*}(\rho)$, so $\tilde{\Delta}(\rho^*, A) = \rho^*\geq 4\rho^*/5$ (because $B^*=(\partial \norm{\cdot})_{0} = \Gamma_{t^*}(\rho)$ according to \eqref{eq:sub_diff_norm}). 
Therefore, understanding the sub-differential of the regularization norm matters when one wants to derive statistical bounds depending on the dimension of the low-dimensional structure that contains $t^*$. 
This is something expected since a norm has sparsity inducing power if its sub-differential is a ``large'' subset of the dual sphere at vectors having the sparse structure (see, for instance, the construction of atomic norms in \cite{MR3138795}). 

We now have all the necessary tools to derive statistical bounds for many procedures by applying Theorem~\ref{main:theo}. In each example (given by a convex and Lipschitz loss function and a regularization norm), we just have to compute the complexity function $r_2$, solve a sparsity equation  and check the local Bernstein condition.


\subsection{The minmax MOM logistic LASSO procedure} \label{sec_mom_lasso}
When the dimension $p$ of the problem is large and $\|t^*\|_0 = |\{i \in \{1,\cdots,p  \} : t_i^* \neq 0 \} |$ is small, it  is possible to derive error rate depending on the size of the support of $t^*$ instead of the dimension $p$ by using a $\ell_1$ regularization norm. It leads to the well-known LASSO estimators, see~\cite{tibshirani1996regression,MR2533469}. For the logistic loss function, its minmax MOM formulation is the following. For a given $K\in\{1, \ldots, N\}$ and $\lambda>0$, the minmax MOM logistic LASSO procedure is defined by
	\begin{equation*}
	\hat{t}_{\lambda,K} \in \argmin_{t \in \mathbb{R}^p}  \sup_{ \tilde{t} \in \mathbb{R}^p}  \bigg(  \MOM{K}{\ell_t - \ell_{\tilde{t}}  }+ \lambda (\|t\|_1  - \|\tilde{t}\|_1)   \bigg)\enspace,
	\end{equation*} 
	with the logistic loss function defined as $\ell_t (x,y)= \log(1+\exp(-y \inr{x, t}) )$ for all $t, x\in\bR^p$ and $y\in\{\pm1\}$, and with the $\ell_1$ regularization norm defined for all $t\in\bR^p$ by $\|t\|_1 = \sum_{i=1}^p |t_i|$. 
	
%
%
We first compute the complexity function $r_2$. Theorem~\ref{theo:shahar} can be applied to upper bound the Rademacher complexities from \eqref{comp:rad} in that case because the dual norm of $\ell_1$-norm (i.e the $\ell_{\infty}$-norm) is $1$-unconditional with respect to $(e_i)_{i=1}^p$. Then, if $X$ is an isotropic random vector satisfying~\eqref{cond:shahar}, Theorem~\ref{theo:shahar} holds and 
	\begin{align} \nonumber \label{comp_lasso}
	 \mathbb{E}\sup_{t \in \rho B_1^p \cap r B_2^p} \bigg |{\sum_{j \in J} \sigma_j \inr{t,X_j}} \bigg | \leq  c(C_0,M) \sqrt{|J|} w (\rho B_1^p \cap r B_2^p)\enspace,
	\end{align}
	where $B_1^p$ denote the unit ball of the $\ell_1$ norm. From \cite[Lemma 5.3]{LM_sparsity}, we have
	\begin{equation} \label{gaussian_mw_lasso}
	w (\rho B_1^p \cap r B_2^p)  \leq c \left\{\begin{array}{cc}
	r\sqrt{p} & \mbox{ if } r\leq \rho/\sqrt{p}\\
	\rho \sqrt{\log (ep \min(r^2/\rho^2, 1))} & \mbox{ if } r\geq \rho/\sqrt{p}
	\end{array}\right.\enspace.
	\end{equation}
	Therefore, one can take
	\begin{equation} \label{eq:choice_r_2_log_LASSO}
	r_2^2(\gamma, \rho) = c(\gamma,C_0,M) \left\{\begin{array}{cc}
	\frac{p}{N} & \mbox{ if } N \rho^2\geq c(\gamma,C_0,M)\gamma p^2\\
	\rho \sqrt{\frac{1}{N}\log\left(\frac{ep^2}{\rho^2 N}\right) } & \mbox{ if } \log p \leq c(\gamma,C_0,M) N\rho^2 \leq c(\gamma,C_0,M) p^2\\
	\rho \sqrt{\frac{\log p}{N}} & \mbox{ if } \log p\geq c(\gamma,C_0,M) N \rho^2.
	\end{array}\right.\enspace.
	\end{equation}

Let us turn to the local Bersntein assumption. We need to verify Assumption~\ref{ass:bernstein_logistic_loss}. Let $\varepsilon > 0$. If $X$ is an isotropic random vector satisfying~\eqref{cond:shahar} and $C_0 \log(p) \geq 2+\varepsilon$, where $C_0$ is the constant appearing in Theorem~\ref{theo:shahar}, then the point a) of Assumption~\ref{ass:bernstein_logistic_loss} is verified with $C' = c(M,C_0)$. 
For any $x \in \bR^p$, let us write $f^*(x) = \inr{x,t^*}$, where $t^* \in \bR^p$. Let us assume that the oracle is such that
\begin{equation} \label{cond_noise_lasso}
\mathbb P \big( |\inr{X,t^*}| \leq c_0 \big) \geq 1-\frac{1}{2(C')^{(4+2\varepsilon)/\varepsilon}}.
\end{equation}Therefore, if Equation~\eqref{cond_noise_lasso} holds, the local Bernstein Assumption is verified for a constant $A$ depending on $M,C_0$ and $c_0$ given in Proposition~\ref{prop:bernstein_logistic_loss} (since the latter formula is rather complicated, we will keep the notation $A$ all along this section). 

Finally, let us turn to a solution to the sparsity equation for the $\ell_1^p$ norm . The result can be found in \cite{LM_sparsity}. 
	\begin{lemma}\cite[Lemma 4.2]{LM_sparsity} \label{lemma_lasso}.
		Let us assume that $X$ is isotropic. If the oracle $t^{*}$ can be decomposed as $t^{*} = v + u$ with $u \in (\rho/20) B_1^p$ and $100s \leq \big(\rho / \sqrt{C_{K,r}(\rho,A)} \big)^2$ then $\Delta(\rho) \geq (4/5)\rho$, where $s = |\text{supp}(v)|$.
	\end{lemma}

Assume that $t^{*}$ is a $s$-sparse vector, so Lemma~\ref{lemma_lasso} applies. 
We consider two cases depending on the values of $K$ and $N r_2^2(\gamma, \rho^*)$. 
When $C_{K,r}(\rho^*,A) = r_2^2(\gamma,\rho^*)$ -- which holds when $K\leq c(c_0,C_0,M) N r_2^2(\gamma, \rho^*)$ --  Lemma~\ref{lemma_lasso} shows that $\rho^{*} =  c(c_0,M,C_0) s \sqrt{\log{(ep/s)}/N}$ satisfies the sparsity equation. 
For these values, the value of $r_2$ given in \eqref{eq:choice_r_2_log_LASSO} yields 
\[
r_2^2(\gamma, \rho^{*}) = c(c_0,M,C_0, \gamma) \frac{s\log(ep/s) }{N}\enspace.
\] 
Now, if $C_{K,r}(\rho,A) = c(A,L) K/N$  -- which holds when $K\geq c(c_0,C_0,M) N r_2^2(\gamma, \rho^*)$--  we can take $\rho^{*} = c(c_0,M,C_0) \sqrt{sK/N}$.  
Therefore, Theorem~\ref{main:theo} applies with 
\[
\rho^{*} = c(c_0,M,C_0)  \max(s \sqrt{\log{(ep/s)}/N},\sqrt{sK/N})\enspace.
\] 
Finally from Remark~\ref{remark:ber}, note that is necessary to have $N \geq c\log(ep/s)$, where $c>0$ is an absolute constant in order to have $A$ like a constant in Proposition~\ref{prop:bernstein_logistic_loss}.

\begin{Theorem}\label{th:lasso} Let $\varepsilon>0$ and $(X,Y)$ be a random variable taking values in $\bR^p \times \{\pm1\}$, where $X$ is an isotropic random vector such that for all $1\leq j \leq p$  and  $1\leq q \leq C_0 \log(p)$, $\norm{\inr{X, e_j}}_{L_q} \leq M \sqrt{q}$ with $C_0 \log(p) \geq 2+ \varepsilon$. 
Let $f^*: x \in \bR^p  \mapsto \inr{x,t^*}$ be the oracle where $t^* \in \bR^p$ is $s$-sparse. Assume also that the oracle satisfies \eqref{cond_noise_lasso}. Assume that $(X,Y),(X_i,Y_i)_{i\in\cI}$ are i.i.d distributed and $N \geq cs \log(ep/s)$.
Let $K \geq 7|\cO|/3$. 
With probability larger than $1-2\exp(- cK )$, the minmax MOM logistic LASSO estimator $\hat{t}_{\lambda, K}$ with 
	\begin{equation*}
	\lambda=c(c_0,M,C_0)\max \bigg (\sqrt{\frac{\log(ep/s)}{N}}, \sqrt{\frac{K}{sN}} \bigg)
	\end{equation*} satisfies
\begin{gather*}
\|\hat{t}_{\lambda,K} - t^{*}\|_1 \leq c(c_0,M,C_0)\max \bigg( s \sqrt{\frac{\log(ep/s)}{N}}, \sqrt{s}\sqrt{\frac{K}{N}} \bigg), \\
 \|\hat{t}_{\lambda,K} - t^{*}\|_2^2 \leq  c(c_0,M,C_0)\max \bigg( \frac{K}{N},  s\frac{\log(ep/s)}{N} \bigg)\enspace,\\
 P\cL_{\hat{f}_{\lambda,K}} \leq c(c_0,M,C_0) \max \bigg( \frac{K}{N},  s\frac{\log(ep/s)}{N} \bigg)\enspace.
\end{gather*}
\end{Theorem}
For $K \leq  c(c_0,M,C_0)  s \log(ep/s)$, the upper bound on the estimation risk and excess risk matches the minimax rates of convergence for $s$-sparse vectors in $\bR^p$. It is also possible to adapt in a data-driven way to the best $K$ and $\lambda$ by using a Lepski's adaptation method such as in \cite{devroye2016sub,lecue2017learning,lecue2017robust,ChiLecLer:2018,chinot2019robust}. This step is now well understood, it is not reproduced here.

\subsection{The minmax MOM logistic SLOPE}

In this section, we study the minmax MOM estimator with the logistic loss function and the SLOPE regularization norm. Given $\beta_1 \geq \beta_2 \geq \cdots \geq \beta_p >0$, the SLOPE norm (see~\cite{MR3418717}) is defined for all $t\in\bR^p$ by
	\begin{equation*}
 \|t\|_{\text{SLOPE}} = \sum_{i=1}^p \beta_i t_i^{*}\enspace,
	\end{equation*}where $(t_i^{*})_{i=1}^p$ denotes the non-increasing re-arrangement of $(|t_i|)_{i=1}^p$. The SLOPE norm coincides with the $\ell_1$ norm when $\beta_j = 1 $ for all $j=1 , \cdots, p$. 

Given $K\in\{1, \ldots,N\}$ and $\lambda>0$, the minmax MOM logistic SLOPE procedure is 
	\begin{equation}\label{eq:minmax_mom_slope}
	\hat{t}_{\lambda,K} \in \argmin_{t \in \mathbb{R}^p}  \sup_{ \tilde{t} \in \mathbb{R}^p}  \bigg(  \MOM{K}{\ell_t - \ell_{\tilde{t}}  }+ \lambda (\|t\|_{\text{SLOPE}}  - \|\tilde{t}\|_{\text{SLOPE}})   \bigg)\enspace,
	\end{equation} 
	where  $\ell_t:(x,y)\in\bR^p\times\{-1, 1\}= \log(1+\exp(-y \inr{x,t}))$ for all $t\in\bR^p$. 
	
	Let us first compute the complexity function $r_2$. If $V \subset \bR^p$ is closed under permutations and reflections (sign-changes)-- which is the case for $B_{SLOPE}^p$, the unit ball of the SLOPE norm -- then $\sup_{v \in V} |\inr{\cdot,v}|$ is $1$-unconditional. Therefore, the dual norm of $\|\cdot\|_{SLOPE}$ is $1$-unconditional and Theorem \ref{theo:shahar} applies provided that $X$ is isotropic and verifies \eqref{cond:shahar}. By \cite[Lemma 5.3]{LM_sparsity}, we have
	\begin{align} \label{res1_slope}
	\nonumber \mathbb{E} \sup_{t \in \rho B_{\text{SLOPE}}^p\cap rB_2^p} \bigg| \sum_{i \in J} \sigma_i \inr{X_i,t}  \bigg| & \leq c(C_0,M)  \sqrt{|J|}  w (\rho B_{\text{SLOPE}}^p \cap r B_2^p)\\
	& \leq c(C_0,M) \sqrt{|J|} \left\{\begin{array}{cc}
	r\sqrt{p} & \mbox{ if } r \leq \rho/\sqrt{p}\\
	\rho & \mbox{ if } r\geq \rho/\sqrt{p}
	\end{array}\right.
	\end{align}
It follows that 
\begin{equation*}
r_2^2(\gamma,\rho) = c(C_0,\gamma,M) \left\{\begin{array}{cc}
\frac{p}{N} & \mbox{ if } p \leq c(C_0,\gamma,M) \rho\sqrt{N}\\
\frac{\rho}{\sqrt{N}} & \mbox{ if } p \geq c(C_0,\gamma,M)  \rho\sqrt{N}.
\end{array}\right.
\end{equation*}

Let us turn to the local Bernstein Assumption. Since the loss function is the same as the one used in Section~\ref{sec_mom_lasso}, the local Bernstein assumption holds if there exists $c_0 >0$ such that 
\begin{equation}\label{eq:oracle_bounded_2} 
\mathbb P \big( |\inr{X,t^*}| \leq c_0 \big) \geq 1-\frac{1}{2(C')^{(2+2\varepsilon)/\varepsilon}}
\end{equation}
where $C' = c(M,C_0)$ is a function of $M$ and $C_0$ only. The constant $A$ in the Bernstein condition depends on $c_0,C_0$ and $M$. As for the LASSO, since the formula of $A$ is complicated (given in Proposition~\ref{prop:bernstein_logistic_loss}), we write $A$ all along this section but we assume that $r_2(\gamma, \rho^*)(2C')^{(2+\varepsilon)/\varepsilon} \leq c_0/2$ so that $A$ can be considered like an absolute constant (depending only on $c_0$). This condition is equivalent to assuming $N\gtrsim s \log(ep/s)$.

A solution to the sparsity equation relative to the SLOPE norm can be found in \cite{LM_sparsity}. We recall this result here.
	\begin{lemma}\cite[Lemma~4.3]{LM_sparsity}\label{se_slope}
		Let $1\leq s \leq p$ and set $\mathcal{B}_s = \sum_{i\leq s} \beta_i/\sqrt{i}$. If $t^{*}$ can be decomposed as $t^{*} = u +v$ with $u \in (\rho/20)B_{\text{SLOPE}}^p$ and $v$ is $s$-sparse and if $40 \mathcal{B}_s \leq \rho/\sqrt{C_{K,r}(\rho,A)}$ then $\Delta(\rho) \geq 4\rho/5$. 
	\end{lemma}
Assume that $t^*$ is exactly $s$-sparse, so that Lemma~\ref{se_slope} applies. 
We consider two cases depending on $K$. Consider the case where $K\leq c(c_0,C_0,M) N r_2^2(\gamma, \rho^*)$, so $\sqrt{C_{K,r}(\rho^*,A)} = r_2(\gamma,\rho^*)$. 
For $\beta_j  = c \sqrt{\log(ep/j)}$, one may show that $\mathcal{B}_s = c \sqrt{s \log(ep/s)} $ (see \cite{bellec2016slope,LM_sparsity}). From~\eqref{res1_slope} and Lemma~\ref{se_slope}, it follows that we can choose
\begin{equation}\label{eq:comp_r2_rho}
\rho^{*} = c(c_0,M,C_0) s \frac{\log(ep/s)}{\sqrt{N}} \quad \mbox{and thus} \quad r_2^2(\gamma, \rho^{*}) = c(c_0,M,C_0)  \frac{s\log(ep/s)}{N}\enspace.
\end{equation}
For $C_{K,r}(\rho,A) =  c(c_0,M,C_0)  K/N$ holding when $K\geq c(c_0,C_0,M) N r_2^2(\gamma, \rho^*)$, we take $\rho^* = c(c_0,C_0,M)   \sqrt{sK/N}$ satisfying the sparsity equation. We can therefore apply Theorem~\ref{main:theo} for
\begin{equation*}
	\rho^{*} = c(c_0,M,C_0)  \max(s \sqrt{\log{(ep/s)}/N},\sqrt{sK}/\sqrt{N}) \enspace.
\end{equation*} 

\begin{Theorem}\label{th:slope}
	 Let $\varepsilon >0$ and $(X,Y)$ be random variable with values in $\bR^p \times \{\pm1\}$ such that $X$ is an isotropic random vector such that for all $1\leq j \leq p$  and  $1\leq q \leq C_0 \log(p)$, $\norm{\inr{X, e_j}}_{L_q} \leq M \sqrt{q}$ with $C_0 \log(p) \geq 2+ \varepsilon$. 
	 Let $f^*: x \in \bR^p  \mapsto \inr{x,t^*}$ be the oracle where $t^* \in \bR^p$ is $s$-sparse. Assume also that the oracle satisfies \eqref{cond_noise_lasso}.
	 Assume that $(X,Y), (X_i,Y_i)_{i\in\cI}$ are i.i.d and $N \geq cs \log(ep/s)$.
	 Let $K \geq 7|\cO|/3$. 
	 Let $\hat t_{\lambda,K}$ be the minmax MOM logistic Slope procedure introduced in \eqref{eq:minmax_mom_slope} for the choice of weights $\beta_j = \sqrt{\log(ep/j)}, j=1, \ldots, p$ and regularization parameter $\lambda =  c(c_0,M,C_0) \max(1/\sqrt{N},\sqrt{K/(sN)})$.  
	 With probability larger than $1-2\exp(- cK )$,
	\begin{gather*}
	\|\hat{t}_{\lambda,K} - t^{*}\|_{\text{SLOPE}} \leq  c(c_0,M,C_0)  \max \bigg( s \sqrt{ \frac{\log(ep/s)}{N}}, \sqrt{s}\sqrt{\frac{K}{N}} \bigg) ,\\ \|\hat{t}_{\lambda,K} - t^{*}\|_2^2 \leq   c(c_0,M,C_0)  \max \bigg( \frac{K}{N},  s\frac{\log(ep/s)}{N} \bigg)  \enspace,\\
	 P\cL_{\hat{t}_{\lambda,K}} \leq    c(c_0,M,C_0)  \max \bigg( \frac{K}{N},  s\frac{\log(ep/s)}{N} \bigg)\enspace.
	\end{gather*}
\end{Theorem} 
For $K \leq  c(c_0,M,C_0)   s \log(ep/s)/N$, the parameter $\lambda$ is independent from the unknown sparsity $s$ and these bounds match the minimax rates of convergence over the class of $s$-sparse vectors in $\bR^p$ without any restriction on $s$ \cite{bellec2016slope}. 
Ultimately, one can use a Lepski's adaptation method to chose in a data-driven way the number of blocks $K$ as in \cite{lecue2017robust} to achieve these optimal rates without prior knowledge on the sparsity $s$.

\subsection{The minmax MOM Huber Group-Lasso} \label{log_gl}
In this section, we consider regression problems where $\cY=\bR$. 
We consider group sparsity as notion of low-dimensionality for $t^*$. 
This setup is particularly useful when features (i.e. coordinates of $X$) are organized by blocks, as when one constructs dummy variables from a categorical variable. 

The regularization norm used to induce this type of  ``structured sparsity'' is called the Group LASSO (see, for example~\cite{yang2015fast} and~\cite{meier2008group}). 
It is built as follows: let $G_1,\cdots,G_M$ be a partition of $\{ 1,\cdots,p \}$ and define, for any $t\in \mathbb{R}^p$
\begin{equation}
\|t\|_{\text{GL}} = \sum_{k=1}^M \|t_{G_k}\|_2 \enspace.
\end{equation}
Here, for all $k=1, \ldots, M$, $t_{G_k}$ denotes the orthogonal projection of $t$ onto the linear $\text{Span}(e_i, i \in G_k)$ -- $(e_1, \ldots, e_p)$ being the canonical basis of $\bR^p$.

The estimator we consider is the minmax MOM Huber Group-LASSO defined, for all $K\in\{1,\cdots,N\}$ and $\lambda >0$, by
\begin{equation*}
\hat{t}_{\lambda,K} \in \argmin_{t \in \mathbb{R}^p}  \sup_{ \tilde{t} \in \mathbb{R}^p}  \bigg(  \MOM{K}{\ell_t - \ell_{\tilde{t}}  }+ \lambda (\|t\|_{GL}  - \|\tilde{t}\|_{GL})   \bigg)\enspace,
\end{equation*} 
where $t\in\bR^p\to\ell_t$ is the Huber loss function with parameter $\delta >0$ defined as
\[
\ell_t(X_i,Y_i) =  
\begin{cases}
\frac{1}{2}(Y_i-\inr{X_i,t})^2&\text{ if }|Y_i- \inr{X_i,t}| \leq \delta\\
\delta|Y_i-\inr{X_i,t}|-\frac{\delta^2}{2}&\text{ if }|Y_i-\inr{X_i,t}| > \delta
\end{cases}\enspace.
\]In particular, it is a Lipschitz loss function with $L=\delta$. 
Estimation bounds and oracle inequalities satisfied by $\hat{t}_{\lambda,K}$ follow from Theorem~\ref{main:theo} as long as we can compute the complexity function $r_2$, we verify the local Bernstein Assumption and we find a radius $\rho^*$ satisfying the sparsity equation. 
We now handle these problems starting with the computation of the complexity function $r_2$. 

The dual norm of $\|\cdot\|_{\text{GL}}$ is $z\in\bR^p\to\|z\|^{*}_{\text{GL}} = \max_{1\leq k \leq M} \|z_{G_k}\|_2$, it is not $\kappa$-unconditional with respect to the canonical basis $(e_i)_{i=1}^p$ of $\mathbb{R}^p$ for some absolute constant $\kappa$, so Theorem \ref{theo:shahar} does not apply directly. Therefore, in order to avoid long and technical materials on the rearrangement of empirical means under weak moment assumptions for the computation of the local Rademacher complexity from \eqref{comp:rad}, we simply assume that the design vectors $(X_i)_{i\in\cI}$ are $L_0$-subgaussian and isotropic: for all $i\in\cI$, all $t\in\bR^p$ and all $q\geq1$
\begin{equation}\label{eq:subgaussian_design}
\norm{\inr{X_i,t}}_{L_q}\leq L_0 \sqrt{q}\norm{\inr{X_i,t }}_{L_2} \mbox{ and } \norm{\inr{X_i,t }}_{L_2}  = \norm{t}_2.
\end{equation}In that case,  a direct chaining argument allows to bound Rademacher processes by the Gaussian processes (see \cite{MR3184689} for chaining methods):
\begin{align*}
\mathbb{E}\sup_{t \in \rho B_{\text{GL}}^p \cap r B^p_2} \bigg | \sum_{j \in J} \sigma_j \inr{t,X_j} \bigg |  \leq  c(L_0)\sqrt J w(\rho B_{\text{GL}}^p \cap r B_2^p )\enspace.
\end{align*}
Here, $B_{\text{GL}}^p$ is the unit ball of $\|\cdot\|_{\text{GL}}$, $w(\rho B_{\text{GL}}^p \cap r B_2^p )$ is the Gaussian mean width of the interpolated body $\rho B_{\text{GL}}^p \cap r B_2^p$. It follows from the proof of Proposition~6.7 in \cite{bellec2017towards} that when the $M$ groups $G_1, \ldots, G_M$ are all of same size $p/M$ we have
\begin{align*}
w(\rho B_{\text{GL}}^p \cap r B_2^p) \leq \left\{
\begin{array}{cc}
c \rho \sqrt{\frac{p}{M} + \log\left(\frac{Mr^2}{\rho^2}\right)} & \mbox{ if } 0<\rho\leq r\sqrt{M}\\
c r \sqrt{p} & \mbox{ if } \rho\geq r\sqrt{M}
\end{array}
\right.\enspace.
\end{align*} 
This yields
\begin{equation} \label{r:GL}
r_2^2(\gamma, \rho) =  c(\delta,L_0,\gamma) \left\{
\begin{array}{cc}
 \frac{\rho}{\sqrt N} \sqrt{\frac{p}{M} + \log\left(\frac{Mr^2}{\rho^2}\right)} & \mbox{ if } 0<c(\delta,L_0,\gamma)\frac{\rho}{r}\leq \sqrt{M}\\
 \frac{r}{ \sqrt N} \sqrt{p} & \mbox{ if } c(\delta,L_0,\gamma) \frac{\rho}{r}\geq \sqrt{M}
\end{array}
\right. \enspace.
\end{equation} 

Let us now turn to the local Bernstein Assumption. We need to verify Assumption~\ref{ass:bernstein_huber_loss}. As we assumed that the design vectors $(X_i)_{i \in \cI}$ are isotropic and $L_0$-subgaussian, it is clear that the point a) in Assumption~\ref{ass:bernstein_huber_loss} holds with $C' = L_0$. Let us take $\varepsilon = 2$ (another choice would only change the constant). 
For the point b), we assume that there exists $\alpha >0 $ such that, for all $x\in\cX$ and all $z\in \mathbb{R}$ satisfying $ |z-f^{*}(x) | \leq 2 L_0 ^2   \sqrt{C_{K,r}(\rho,4/\alpha)}$, $F_{Y|X=x}(z+\delta) - F_{Y|X=x}(z- \delta)\geqslant \alpha$. Under these conditons, the local Bernstein Assumption is verified for $A = 4/\alpha$ according to Proposition~\ref{prop:bernstein_huber_loss}. We will assume that $C_{K,r}(\rho^*,4/\alpha)\leq c$ for some absolute constant $c$ so that $\delta$ and $\alpha$ can be taken like absolute constant.  Condition ``$C_{K,r}(\rho^*,4/\alpha)\leq c$'' is satisfied when $N\gtrsim c s \log(ep/s)$.

Finally, we turn to the sparsity equation. The following lemma is an extension of Lemma~\ref{lemma_lasso} to the Group Lasso norm. 
\begin{lemma} \label{spars:GL}
	Assume that $X$ is isotropic. Assume that $t^{*} = u + v$ where $\|u\|_{\text{GL}} \leq \rho/20$ and $v$ is group-sparse i.e $v_{G_k}=0$ for all $k\notin I$ for some $I\subset \{1, \ldots,M\}$. If $100 |I| \leq (\rho/\sqrt{C_{K,r}(\rho,4/\alpha)})^2$, then $\Delta(\rho) \geq 4\rho /5 $. 
\end{lemma}
\proof
Let us define $r(\rho) := \sqrt{C_{K,r}(\rho,4/\alpha)}$ and recall that 
\begin{equation*}
\tilde{\Delta}(\rho, 4/\alpha) = \inf_{w\in \rho S_{GL}\cap r(\rho) B_2^p}\sup_{z^*\in\Gamma_{t^*}(\rho)}\inr{z^*, w}\enspace.
\end{equation*}
Here, $S_{GL}$ is the unit sphere of $\norm{\cdot}_{GL}$ and $\Gamma_{t^*}(\rho)$ is the union of all sub-differentials $(\partial\norm{\cdot}_{GL})_v$ for all $v\in t^* + (\rho/20)B_{GL}^p$. We want to find  a condition on $\rho>0$  insuring that $\tilde{\Delta}(\rho, 4/\alpha)\geq 4\rho/5$. 

Let $w$ be a vector in $\mathbb{R}^p$ such that $\|w\|_{GL} = \rho$ and $\|w\|_2 \leq r(\rho)$. We construct $z^*\in\bR^p$ such that $z^*_{G_k} = w_{G_k}/\norm{w_{G_k}}_2$ if $k\notin I$ (so that $\inr{z^{*}_{G_k}, w_{G_k}} = \|w_{G_k}\|_2$ for all $k \notin I$) and $z^*_{G_k} = v_{G_k} / \norm{v_{G_k}}_2$ if $k\in I$ (so that $\inr{z_{G_k}, v_{G_k}}=\norm{v_{G_k}}_2$ for all $k \in I$). We have $\norm{z^*_{G_k}}_2=1$ for all $k\in[M]$, so $\norm{z^*}_{GL}^*=1$ (i.e. $z^*$ is in the dual sphere of $\norm{\cdot}_{GL}$) and $\inr{z^*, v}=\norm{v}_{GL}$ (i.e. $z^*$ is norming for $v$). Therefore, it follows from \eqref{eq:sub_diff_norm} that $z^*\in(\partial \norm{\cdot}_{GL})_v$. Moreover, $\|u\|_{GL} \leq \rho/20$ hence $v\in t^*+(\rho/20)B_{GL}^p$ and so $z^*\in\Gamma_{t^*}(\rho)$. Furthermore, for this choice of sub-gradient $z^*$, we have
	\begin{align*}
	\inr{z^{*},w} & = \sum_{k \in I } \inr{z^{*}_{G_k},w_{G_k}} +  \sum_{k \notin I } \inr{z^{*}_{G_k},w_{G_k}} \geq -  \sum_{k \in I } \|w_{G_k}\|_2 +  \sum_{k \notin I } \norm{w_{G_k}}_2 \\
	& =  \sum_{k=1}^M \|w_{G_k}\|_2 - 2  \sum_{k \in I } \|w_{G_k}\|_2 \geq \rho - 2 \sqrt{|I|}r(\rho)\enspace.
	\end{align*} 
In the last inequality, we used that $\norm{w}_{GL}=\rho$ and  that  
\begin{equation*}
\sum_{k \in I } \|w_{G_k}\|_2  \leq \sqrt{|I|} \sqrt{\sum_{k \in I } \|w_{G_k}\|_2^2}\leq \sqrt{|I|}\norm{w}_2 \leq \sqrt{|I|} r(\rho).
\end{equation*} Then $\inr{z^{*},w } \geq 4\rho/5$ when $\rho - 2\sqrt{|I|} r(\rho) \geq 4\rho/5$ which happens to be true when $100|I|\leq  (\rho/r(\rho))^2$. 
\endproof

Assume that $t^*$ is exactly $s$-group sparse, so Lemma~\ref{spars:GL} applies. 
We consider two cases depending on the value of $K$. 
When $K\leq c(L_0,\alpha, \delta) N r_2^2(\gamma, \rho^*)$, $\sqrt{C_{K,r}(\rho^*,4/\alpha)} = r_2(\gamma,\rho^*)$. 
By Lemma~\ref{spars:GL} and~\eqref{r:GL}, it follows that (for equal size blocks), one can choose 
\begin{align} \label{res_GL} 	\rho^{*} =  c(L_0,\alpha, \delta)   \frac{s}{\sqrt{N}} \sqrt{\frac{p}{M} + \log M} \quad \mbox{and thus} \quad 	r^2( \gamma, \rho^{*}) = c(L_0,\alpha, \delta)    \frac{s}{N} \bigg( \frac{p}{M} + \log M \bigg) \enspace.
\end{align}
This result has a similar flavor as the one for the Lasso. 
The term $s^\prime = sp/M$ equals \emph{block sparsity} $\times$ \emph{size of each blocks}, i.e to the total number of non-zero coordinates in $t^*$: $s^\prime = \norm{t^*}_0$. 
Replacing the sparsity $s^\prime$ by $sp/M$ in Theorem~\ref{th:lasso}, we would have obtained  $\rho^* =  c(L_0,\alpha, \delta)  (sp/M) \sqrt{\log(p)/N}$ which is larger than the bound obtained for the Group Lasso in Equation~\eqref{res_GL}. 
It is therefore better to induce the sparsity by blocks instead of just coordinate-wise when we are aware of such block-structured sparsity. 
In the other case, when $K\leq c(L_0,\alpha, \delta)  N r_2^2(\gamma, \rho^*)$, we have $\sqrt{C_{K,r}(\rho^*,4/\alpha)}  = c(L_0,\alpha, \delta)   \sqrt{K/N}$ and so one can take $\rho^* = c(L_0,\alpha, \delta)  \sqrt{sK/N}$.
We can therefore apply Theorem~\ref{main:theo} with
\begin{equation*}
\rho^* = c(L_0,\alpha, \delta)  \max \bigg( \frac{s}{\sqrt N} \sqrt{ \frac{p}{M}+\log(M)} , \sqrt{s}\sqrt{\frac{K}{N}}  \bigg)\enspace.
\end{equation*} 

\begin{Theorem}\label{th:gl}Let $(X,Y)$ be a random variables with values in $\bR^p\times \bR$ such that $Y \in L_1$ and $X$ is an isotropic and $L_0$-subgaussian random vector in $\bR^p$. 
Assume that $(X,Y), (X_i,Y_i)_{i \in \cI}$ are i.i.d.
Let $f^*(\cdot) = \inr{t^*,\cdot}$ for some $t^*\in\bR^p$ which is $s$-group sparse with respect to equal-size groups $(G_k)_{k=1}^M$. 
Let $K \geq 7|\cO|/3$ and $N \geq cs(p/M+\log(M))$.
Assume that there exists $\alpha >0$ such that, for all $x \in \bR^p$ and all $z \in \bR$ satisfying $|z-\inr{t^*,x}| \leq 2L_0^2 \sqrt{C_{K,r}(2 \rho^*,4/\alpha)}$, $F_{Y | X=x} (\delta + z) - F_{Y | X = x}(z-\delta) \geq \alpha$ (where $F_{Y|X=x}$ is the cumulative ditribution function of $Y$ given $X=x$). 
With probability larger than $1-2\exp(- cK)$, the MOM Huber group-LASSO estimator $\hat t_{\lambda,K}$ for 
	\begin{equation*}
	\lambda = c(L_0,\alpha, \delta)  \max \bigg( \frac{1}{\sqrt{N}}\sqrt{\frac{p}{M} +\log M}, \sqrt{\frac{K}{sN}} \bigg)
	\end{equation*} 
	satisfies
	\begin{gather*}
	\|\hat{t}_{\lambda,K} - t^{*}\|_{\text{GL}}  \leq c(L_0,\alpha, \delta)  \max \bigg( \frac{s}{\sqrt N} \sqrt{ \frac{p}{M}+\log(M)} , \sqrt{s}\sqrt{\frac{K}{N}}  \bigg), \\
	 \|\hat{t}_{\lambda,K} - t^{*}\|_2^2 \leq  c(L_0,\alpha, \delta)  \max \bigg(  \frac{s}{N} \left( \frac{p}{M}+ \log(M) \right), \frac{K}{N} \bigg),\\
	P\cL_{\hat{t}_{\lambda,K}}  \leq c(L_0,\alpha, \delta)  \max \bigg(  \frac{s}{N} \left( \frac{p}{M}+ \log(M) \right), \frac{K}{N} \bigg)  \enspace.
	\end{gather*}
\end{Theorem} 
For $K \leq  c(L_0,\alpha, \delta)  s(p/M + \log M)$, the regularization parameter $\lambda$ is independent from the unknown group sparsity $s$ (the choice of $K$ can be done in data-driven way using either a Lepski method or a MOM cross validation as in \cite{lecue2017robust}).  In the ideal i.i.d. setup (with no outliers), the same result holds for the RERM as we assumed that the class $F-f^*$ is $L_0$-subgaussian and for the choice of regularization parameter $\lambda = c(L_0,\alpha, \delta)  ( \sqrt{p/(NM)} + \sqrt{\log(M)/N})$. The minmax MOM estimator has the advantage to be robust up to $c(L_0,\alpha, \delta)   s(p/M + \log M)$ outliers in the dataset.

 \subsection{Huber regression with total variation penalty} \label{app_tv}
In this section, we investigate another type of structured sparsity induced by the total variation norm. Given $t \in \R^p$, the Total Variation norm~\cite{osher2005iterative} is defined as
  \begin{equation}
  \|t\|_{TV}  = |t_1| + \sum_{i=1}^{p-1} |t_{i+1}-t_i|  = \|Dt \|_1, \quad \text{where}\quad D = \begin{bmatrix}
  	1 & 0 &\cdots & 0 & 0\\
  	-1 & 1 &\cdots & 0  & 0\\
  	\cdot & \cdot & \cdots & \cdot & \cdot\\
  	\cdot & \cdot & \cdots & \cdot & \cdot\\
  	0 & 0 & \cdots  & -1 & 1 
  \end{bmatrix}\in\bR^{p\times p}\enspace.
  \end{equation} 
  The total variation norm favors vectors such that their ``discrete gradient $Dt$ is sparse'' that is piecewise constant vectors $t$. 

The estimator considered in this section is the minmax MOM Huber TV regularization defined for all $\lambda >0$ and $K \in \{ 1,\cdots, N  \}$ as
  \begin{equation*}
  \hat{t}_{\lambda,K} \in \argmin_{t \in \mathbb{R}^p}  \sup_{ \tilde{t} \in \mathbb{R}^p}  \bigg(  \MOM{K}{\ell_t - \ell_{\tilde{t}}  }+ \lambda (\|t\|_{TV}  - \|\tilde{t}\|_{TV})   \bigg)\enspace,
  \end{equation*} 
  where the loss $\ell$ is the Huber loss: for $\delta >0$,
  \[
  \ell_t(X_i,Y_i) =  
  \begin{cases}
  \frac{1}{2}(Y_i-\inr{X_i,t})^2&\text{ if }|Y_i- \inr{X_i,t}| \leq \delta\\
  \delta|Y_i-\inr{X_i,t}|-\frac{\delta^2}{2}&\text{ if }|Y_i-\inr{X_i,t}| > \delta
  \end{cases}\enspace.
  \]

  Statistical bounds for $\hat{t}_{\lambda,K}$ follows from Theorem~\ref{main:theo} and the computation of $r_2$, $\rho^*$ and the study of the local Bernstein assumption. We start with the computation of the complexity function $r_2$. Simple computations yield that the dual norm of $\|\cdot\|_{TV}$ is $z \in \bR^p \mapsto \|z\|_{TV}^* = \|(D^{-1})^Tz\|_{\infty} = \max_{1 \leq k \leq p} |\sum_{i=1}^k z_i| $ which is not $\kappa$-unconditional with respect to the canonical basis $(e_i)_{i=1}^p$ of $\bR^p$ for some absolute constant $\kappa$. Therefore, Theorem \ref{theo:shahar} does not apply directly. To upper bound the Rademacher complexity from~\eqref{comp:rad}, we assume that the design vectors $(X_i)_{i \in \cI}$ are $L_0$-subgaussian and isotropic (see Equation~\eqref{eq:subgaussian_design}) as in Section~\ref{log_gl}. A direct chaining argument allows to bound the Rademacher complexity by the Gaussian mean width (see \cite{MR3184689} for chaining methods):
 \begin{align*}
 \mathbb{E}\sup_{t \in \rho B_{\text{TV}}^p \cap r B^p_2} \bigg | \sum_{j \in J} \sigma_j \inr{t,X_j} \bigg |  \leq  c(L_0)\sqrt J w(\rho B_{\text{TV}}^p \cap r B_2^p )
 \end{align*}
  \begin{Lemma}
  	For any $\rho,r >0$ such that $\rho \geq r$, $w( \rho B_{\text{TV}}^p \cap r B_2^p ) \leq c \sqrt{r\rho} p^{1/4}$ 
  \end{Lemma}
  \begin{proof}
  Let $\rho,r >0$ be such that $\rho \geq r$. From a simple chaining argument, it follows that 
  \begin{align*}
  w( \rho B_{\text{TV}}^p \cap r B_2^p ) & \leq \int_0^r \sqrt{\log N( \rho B_{TV}\cap r B_2^p, \epsilon B_2^p)}d\epsilon\leq \int_0^r \sqrt{\log N( B_{TV}\cap  B_\infty^p, (\epsilon/\rho) B_2^p)}d\epsilon \enspace,
  \end{align*}
  where $N(B_{TV}\cap B_\infty^p, \epsilon B_2^p)$ represents the number of translates of $\epsilon B_2^p$ needed to cover $B_{TV}\cap B_\infty^p$.
 From Lemma 4.3 in ~\cite{van2020logistic} 
  $$
  N(B_{TV}\cap  B_\infty^p, \epsilon B_2^p) \leq c\sqrt p/\epsilon \enspace,
  $$
  it follows that
  $$
   w( \rho B_{\text{TV}}^p \cap r B_2^p ) \leq \int_0^r \sqrt{\frac{\rho \sqrt p}{\epsilon}}d\epsilon \leq c \sqrt{r\rho} p^{1/4} \enspace. 
  $$
  \end{proof}
So one can take
\begin{equation*}
r_2^3(\gamma,\rho) = c(\delta,L_0,\gamma) \frac{\rho\sqrt p}{n}.
\end{equation*}  
  
  Let us now turn to the local Bernstein Assumption. The loss function and the model being the same as the ones in Section~\ref{log_gl} the Bernstein Assumption is verified with a constant $A = 4/\alpha$, if there exists a constant $\alpha >0$ such that for all $x\in\cX$ and all $z\in \mathbb{R}$ satisfying $ |z-f^{*}(x) | \leq 2 L_0 ^2  \sqrt{C_{K,r}(\rho,4/\alpha)}$, $F_{Y|X=x}(z+\delta) - F_{Y|X=x}(z- \delta)\geqslant \alpha$.  \\

 Let us turn to the sparsity equation.   The following Lemma solves the sparsity equation for the TV regularization.
  \begin{lemma} \label{se_tv}
  	Let us assume that $X$ is isotropic . If the oracle $t^*$ can be decomposed as $t^{*} = v + u $ for $u \in (\rho/20) B_{TV}^p$ and $400 s \leq (\rho/\sqrt{C_{K,r}(\rho,4/\alpha)})^2$, then $\Delta(\rho) \geq  4\rho/5$, where $s = |supp(Dv)|$.
  \end{lemma}
Compared with Lemma~\ref{lemma_lasso}, sparsity in Lemma~\ref{se_tv} is granted on the linear transformation $Dt^*$ (also called discrete gradient of $t^*$) rather than on the oracle $t^*$. 
 \begin{proof}
Let us denote $\sqrt{C_{K,r}(\rho,4/\alpha)} := r(\rho)$. Let us recall that 
\begin{equation*}
\tilde{\Delta}(\rho, 4/\alpha) = \inf_{w\in \rho S_{TV}\cap r(\rho) B_2^p}\sup_{z^*\in\Gamma_{t^*}(\rho)}\inr{z^*, w}
\end{equation*}where $S_{TV}$ is the unit sphere of $\norm{\cdot}_{TV}$ and $\Gamma_{t^*}(\rho)$ is the union of all sub-differentials $(\partial\norm{\cdot}_{TV})_v$ for all $v\in t^* + (\rho/20)B_{TV}^p$. We want to find  a condition on $\rho>0$  insuring that $\tilde{\Delta}(\rho, 4/\alpha)\geq 4\rho/5$.  \\
Recall that the oracle $t^*$ can be decomposed as $t^* = u + v$, where $u \in (\rho/20)B_{TV}$ and thus $\|t^{*} - v\|_{TV} \leq \rho/20$. Let $I$ denote the support of $Dv$ and $s$ its cardinality. Let $I^C$ be the complementary of $I$. Let $w\in\rho S_{TV}^p \cup r(\rho) B_2^p $. \\

We construct $z^* = D^T u^*$, such for all $i$ in $I$, $u_i^* = sign \big((Dv)_i \big)$ and for all $i$ in $I^C$, $u_i^* = sign \big( (Dw)_i \big)$. 
Such a choice of $z^*$ implies that $\inr{z^*,v} = \inr{u^*,Dv} = \sum_{i \in I}  sign\big( (Dv)_i \big) (Dv)_i = \|v\|_{TV}$ i.e $z^*$ is norming for $v$. Moreover, we have $\|z^*\|_{TV}^* = \|(D^{-1})^T z^* \|_{\infty} = \|u^*\|_{\infty} = 1$ hence $z^* \in S_{TV}^*$. Then it follows from~\eqref{eq:sub_diff_norm} that $z^* \in (\partial \|\cdot\|_{TV})_v$ and since $u \in   (\rho/20)B_{TV} $ we have $z^* \in  \Gamma_{t^*}(\rho)$.\\

Now let us denote by $P_I w$ the orthogonal projection of $w$ onto $\text{Span}(e_i, i \in I)$. From the choice of $z^*$ we get
 	\begin{align*}
 	\inr{z^{*},w}  = \inr{D^Tu ^{*}, w} = \inr{u^{*},Dw}  & = \inr{u^{*},P_I Dw } + \inr{u^{*},P_{I^C} Dw } \\
	& \geq - \|P_I Dw\|_1 + \|P_{I^C} Dw \|_1 =  \|Dw\|_1 - 2\|P_I Dw\|_1
 	\end{align*} 
 	Moreover we have $\|P_I Dw\|_1 \leq \sqrt{s} \|P_I Dw\|_2 \leq \sqrt{s} \|Dw\|_2$ and, for $I_p$ the identity matrix, 
 	\begin{equation*}
 	\|Dw\|_2 = \|(I_p + D^-) w \|_2 \leq \| w\|_2 + \|D^- w \|_2 \leq 2 \|w\|_2 \leq  2 r(\rho)\enspace,
 	\end{equation*}
 	where
	\[
	 D^-= \begin{bmatrix}
 	0 & 0 &\cdots & 0 & 0\\
 	-1 & 0 &\cdots & 0  & 0\\
 	\cdot & \cdot & \cdots & \cdot & \cdot\\
 	\cdot & \cdot & \cdots & \cdot & \cdot\\
 	0 & 0 & \cdots  & -1 &0  
 	\end{bmatrix}\in\bR^{p\times p}\enspace.
	\]
	Since $\|Dw\|_1 = \|w\|_{TV} = \rho$, we get $\Delta(\rho)  \geq \rho - 4\sqrt{s} r(\rho) \geq 4\rho/5$  when $\rho \geq 20 \sqrt{s} r(\rho)$.
 \end{proof}

Let us now identify a radius $\rho^*$ satisfying the sparsity equation using Lemma~\ref{se_tv}. We place ourselves under the assumption from Lemma~\ref{se_tv} that is when $t^*$ is such that $Dt^*$ is approximately $s$-sparse. There are two cases to study according to the value of $K$. For the case where $\sqrt{C_{K,r}(\rho^*,4/\alpha)} = r_2(\gamma, \rho^*)$-- which holds when $K\leq c(L_0,\alpha,\delta) N r_2^2(\gamma, \rho^*)$-- , we can take
 \begin{align*}{}
 \rho^* =  c(L_0,\alpha,\delta) \frac{s^{3/4}p^{1/4}}{\sqrt{N}} \quad \mbox{and} \quad r^2_2(\gamma,\rho^*) =  c(L_0,\alpha,\delta) \frac{\sqrt{sp}}{N}. 
 \end{align*}
 For $C_{K,r}(\rho^*,4/\alpha) = c(L_0,\alpha,\delta) K/N$-- which holds when $K\geq c(L_0,\alpha,\delta) N r_2^2(\gamma, \rho^*)$-- we can take $\rho^* = c(L_0,\alpha,\delta) \sqrt{sK/N}$. We can therefore apply Theorem~\ref{main:theo} with
 \begin{equation*}
	 \rho^* = c(L_0,\alpha,\delta) \max\left(\frac{s^{3/4}p^{1/4}}{\sqrt{N}}  ,\sqrt{sK/N} \right) \enspace.
 \end{equation*}
 
To simplify the presentation, we assume that $Dt^*$ is exactly $s$-sparse. We may only assume it is approximatively $s$-sparse using the more involved formalism of Lemma~\ref{se_tv}.
 \begin{Theorem}Let $(X,Y)$ be a random variables with values in $\bR^p\times \bR$ such that $Y \in L_1$ and $X$ is an isotropic and $L_0$-subgaussian random vector in $\bR^p$. Assume that $(X,Y), (X_i,Y_i)_{i \in \cI}$ are i.i.d.
Let $f^*(\cdot) = \inr{t^*,\cdot}$, where $t^*$ is such that $Dt^*$ is $s$-sparse. 
 Let $K \geq 7|\cO|/3$ and $N \geq c s^{3/4}p^{1/4}$.
 Assume that there exists $\alpha >0$ such that, for all $x \in \bR^p$ and all $z \in \bR$ satisfying $|z-\inr{t^*,x}| \leq 2L_0^2 \sqrt{C_{K,r}(2 \rho^*,4/\alpha)}$, $F_{Y | X=x} (\delta + z) - F_{Y | X = x}(z-\delta) \geq \alpha$ (where $F_{Y|X=x}$ is the cumulative distribution function of $Y$ given $X=x$). 
 With probability larger than $1-2\exp(- cK)$, the MOM Huber TV estimator $\hat t_{\lambda,K}$ for 
 	\begin{equation*}
 	\lambda= c(L_0,\alpha,\delta)   \max \bigg( \bigg(\frac{p}{s}\bigg)^{1/4} \frac{1}{\sqrt{N}},\sqrt{ \frac{K}{sN} }  \bigg)
 	\end{equation*} 
 	satisfies
 	\begin{gather*}
 	\|\hat{t}_{\lambda,K} - t^{*}\|_{\text{TV}} \leq c(L_0,\alpha,\delta)  \max   \bigg( \frac{s^{3/4}p^{1/4}}{\sqrt{N}}  , \sqrt{s}\sqrt{\frac{K}{N}} \bigg)   
 		 \\ 
 		 \|\hat{t}_{\lambda,K} - t^{*}\|_2^2 \leq c(L_0,\alpha,\delta) \max \bigg(  \frac{\sqrt{sp}}{N}, \frac{K}{N} \bigg) \enspace,\\
 	P\cL_{\hat{t}_{\lambda,K}} \leq  c(L_0,\alpha,\delta)  \max \bigg(  \frac{\sqrt{sp} }{N}, \frac{K}{N} \bigg) \enspace.
 	\end{gather*}
 \end{Theorem}
Since the Assumptions on the design $X$ imply that the class $F-f^*$ is $L_0$-subgaussian. The minmax MOM estimator has the advantage to be robust up to $\sqrt {sp}$ outliers in the dataset without deteriorating the rate of convergence.

\subsection{Other possible applications}
The fusion of two sparsity structures, namely the Total Variation and $\ell_1$ norms leads to the fused Lasso (see~\cite{tibshirani2005sparsity}) defined for some mixture parameters $\beta, \eta > 0$ for all $t \in \bR^p$ by
\begin{equation*}
\|t\|_{FL}  = \eta \|t\|_1 + \beta \|t\|_{TV} \enspace.
\end{equation*} This type of norm is expected to promote signals having both a small number of non-zero coefficients (thanks to the $\ell_1^p$-norm) and a sparse discrete gradient (thanks to the TV norm) i.e. sparse and constant by blocks signals. It is possible to use our approach to study theoretical guarantees of this estimator. The technical point is the computation of the local Gaussian mean width $w(B_{FL}^p \cap rB_2^p )$, for $r>0$, where $B_{FL}^p$ denotes the unit ball associated with $\|\cdot\|_{FL}$. We may use some trivial bound such as 
\begin{equation}\label{eq:trivial_bound_fused}
w(B_{FL}^p \cap rB_2^p )\leq \min\left( w((1/\eta)B_1^p\cap rB_2^p ), w((1/\beta)B_TV^p\cap rB_2^p \right)
\end{equation}to obtain a result for the Fused LASSO similar to the one obtain for the $\ell_1$-penalty and the TV penalty. However, we believe that a sharper analysis of the local Gaussian mean width $w(B_{FL}^p \cap rB_2^p )$ together with a better understanding of the sparsity inducing power of  $\|\cdot\|_{FL}$ could reveal more interesting phenomena and a better fit of the mixture parameters $\eta$ and $\beta$ than the trivial bound \eqref{eq:trivial_bound_fused} allows. We leave this problem  open for the moment. 



Nevertheless, a take home message is as follows: as soon as we are able to compute the complexity parameter (often directly related to a local Gaussian mean-width), we can apply our approach and establish sharp oracle inequalities. It may however be a difficult problem to get a sharp upper bound on this complexity parameter and the fused lasso is a typical example. 

\section{Simulations}\label{sec:Simu}
This section provides a  simulation study to illustrate our theoretical findings. Minmax MOM estimators are approximated using an alternating  proximal  block gradient descent/ascent with a wisely chosen block of data as in \cite{lecue2017robust}. At each iteration, the block on which the descent/ascent is performed is chosen according to its ``centrality'' (see algorithm~\ref{algo_mom} below). There are so far no theoretical guarantees of convergence of this MOM version of the projected gradient descent/ascent algorithm. However, the aim of this section is to show that it works well in practice. To that end, two examples from high-dimensional statistics are considered 1) Logistic classification with a $\ell_1$ penalization and 2) Huber regression with a Group-Lasso penalization. 

\subsection{Presentation of the algorithm}
Let $\cX = \bR^p$ and let  $F = \{  \inr{t,\cdot}, t \in \bR^p   \}$.  The oracle $f^{*} = \argmin_{f \in F} P\ell_f(X,Y)$ is such that $f^{*}(\cdot) = \inr{t^{*}, \cdot}$ for some $t^*\in\bR^p$. The minmax MOM estimator is defined as
\begin{equation} \label{MOM_lin}
\hat{t}_{\lambda,K} \in \argmin_{t \in \bR^p} \sup_{\tilde{t} \in \bR^p}  MOM_K(\ell_t - \ell_{\tilde{t}}) + \lambda (\|t\| - \|\tilde{t}\|)
\end{equation}
where $\ell$ is a convex and Lipschitz loss function and $\|\cdot\|$ is a norm in $\bR^p$. 

Following the idea of \cite{lecue2017robust}, the minmax problem~\eqref{MOM_lin} is approximated by a proximal block gradient ascent-descent algorithm, see Algorithm~\ref{algo_mom}. 
At each step, one considers the block of data realizing the median and perform an ascent/descent step onto this block. 
The regularization step is obtained via the proximal operator
\begin{equation*}
\text{prox}_{\lambda \|\cdot\|}:x\in\mathbb{R}^p \to \argmin_{y \in \bR^p}  \bigg \{ \frac{1}{2}\|x-y\|_2^2 + \lambda \|y\| \bigg\} \enspace.
\end{equation*}
\RestyleAlgo{boxruled}
\LinesNumbered
\begin{algorithm}[h]  
	\SetAlgoLined
	\KwIn{A number of blocks $K$, initial points $t_0$ and $\tilde{t}_0$ in $\mathbb{R}^p$,  two sequences of step sizes $(\eta_t)_t$ and $(\tilde{\eta}_t)_t$ and $T$ a number of epochs}
	\KwOut{ An approximating solution of the minimax problem \eqref{MOM_lin}}
	\For{$i = 1,\cdots,T$}{
		Construct a random equipartition $B_1 \sqcup \cdots \sqcup B_K$ of $\{ 1, \cdots ,N \} $\\
		Find $k \in [K] $ such that $\text{MOM}_K (\ell_{t_i}-\ell_{\tilde{t}_i})= P_{B_k}(\ell_{t_i}-\ell_{\tilde{t}_i})$\\
		Update:\\
		 \centering $t_{i+1}  = \text{prox}_{\lambda \|\cdot\|} \big( t_i - \eta_i\nabla_t(t\to P_{B_k}\ell_t)_{|t=t_i} \big) $ \\
		 \vspace{0.2cm}
		 \centering $\tilde{t}_{i+1}  = \text{prox}_{\lambda \|\cdot\|} \big( \tilde{t}_i - \tilde{\eta}_i\nabla_{\tilde{t}}(\tilde{t}\to P_{B_k}\ell_{\tilde{t}})_{|\tilde{t}=\tilde{t}_i} \big)$
	}
	\caption{Proximal Descent-Ascent gradient method with median blocks \label{algo_mom}}
\end{algorithm}

To make the presentation simple in Algorithm~\ref{algo_mom}, we have not perform any line search or any sophisticated stopping rule (see, \cite{lecue2017robust} for more involved line search and stopping rules in the setup of minmax MOM algorithms). To compare the statistical and robustness performances of the minmax MOM and  RERM, we perform a proximal gradient descent to approximate the RERM, see Algorithm~\ref{algo_erm} below. 
\RestyleAlgo{boxruled}
\LinesNumbered
\begin{algorithm}[h]  
	\SetAlgoLined
	\KwIn{Initial points $t_0$ in $\mathbb{R}^p$ and a sequence of stepsizes $(\eta_t)_t$}
	\KwOut{ Approximating solution to the RERM estimator.}
	%
	\For{$i = 1, \cdots , T$}{
		\centering $t_{i+1}  = \text{prox}_{\lambda \|\cdot\|} \big( t_i - \eta_i \nabla_t(t \to
		 P_{N}\ell_t)_{|t=t_i} \big) $ \\
	}
	\caption{Proximal gradient descent algorithm \label{algo_erm} }
\end{algorithm}\\
The number of blocks $K$ is chosen by MOM cross-validation (see~\cite{lecue2017robust} for more precision on that procedure). The sequences of stepsizes are constant along the algorithm $(\eta_t)_t := \eta$ and $(\tilde{\eta}_t)_t = \tilde \eta$ and are also chosen by MOM cross-validation.


\subsection{Organisation of the results}
In all simulations, the links between inputs and outputs are given in the regression and classification problems in $\bR^p$ respectively by the following model:
\begin{equation}\label{classif}
\mbox{in regression: }Y =  \inr{X,t^{*}} + \zeta;\qquad  \mbox{in classification: } Y =  \text{sign} \big( \inr{X,t^{*}} + \zeta \big) 
\end{equation}where the distribution of $X$ and $\zeta$ depend on the considered framework:
\begin{itemize}
	\item \textbf{First framework:} $X$ is a standard Gaussian random vector in $\bR^p$ and $\zeta$ is a real-valued standard Gaussian variable  independent of $X$ with variance $\sigma^2$. 
	\item \textbf{Second framework:} $X$ is a standard Gaussian random vector in $\bR^p$ and $\zeta \sim \mathcal T (2)$ (student distribution with $2$ degrees of freedom). This framework is used to verify the robustness w.r.t the noise.
	\item \textbf{Third framework:} $X = (x_1,\cdots,x_p)$ with $x_1,\ldots, x_p \overset{i.i.d.}{\sim} \mathcal T (2)$ and  $\zeta$ is a real-valued standard Gaussian variable  independent of $X$ with variance $\sigma^2$. Here we want to test the robustness w.r.t heavy-tailed design $(X_i)_i$.
	\item \textbf{Fourth framework:}  $X = (x_1,\cdots,x_p)$ with $x_1,\ldots, x_p \overset{i.i.d.}{\sim} \mathcal T (2)$ and $\zeta \sim \mathcal T (2)$. We also corrupt the database with $|\cO|$ outliers which  are such that for all $i\in\cO$, $X_i = (10^5)_{i=1}^p$ and $Y = 1$. Here we verify the robustness w.r.t possible outliers in the dataset.
\end{itemize}
 
In a both first and second frameworks, the RERM and minmax MOM estimators are expected to perform well according to Theorem~\ref{thm:main_subgaussian} and Theorem~\ref{main:theo} even though the noise $\zeta$ can be heavy-tailed. In the third framework, the design vector $X$ is no longer subgaussian, as a consequence Theorem~\ref{thm:main_subgaussian} does not apply and we have no guarantee for the RERM. On the contrary, Theorem~\ref{main:theo} provides statistical guarantees for the  minmax MOM estimators. Nevertheless, it should also be noticed that the study of RERM under moment assumptions on the design can also be performed, see for instance \cite{MR3763780}. In that case, the rates of convergence are still the same but the deviation is only polynomial whereas it is exponential for the minmax MOM estimators. Therefore, in the third example, we may expect similar performance for both estimators but with a larger variance in the results for the RERM. In the fourth framework,  the database has been corrupted by outliers (in both outputs $Y_i$ and inputs $X_i$); in that case, only minmax MOM estimators are  expected to perform well as long as $|\cO|$ is not too large compare with $K$, the number of blocks.

\subsection{Sparse Logistic regression}
Let $\ell$ denote the Logistic loss (i.e. $t\in\bR^p\to \ell_t(x,y)=\log(1+\exp(-y\inr{x,t})), \forall x\in\bR^p, y\in\cY=\{\pm1\}$), and let the $\ell_1$ norm in $\bR^p$ be the regularization norm. Figure~\ref{fig:log} presents the results of our simulations for $N=1000$, $p = 400$ and $s = 30$. In subfigures (a), (b) and (c) the error is the $L_2$ error, which is here $\norm{\hat t_{K,\lambda}^T - t^*}_2$, between the output $\hat t_{K,\lambda}^T$ of the algorithm and the true $t^* \in \bR^p$. In subfigure (d), an increasing number of outliers is added. The error rate is the proportion of misclassification on a test dataset. The stepsizes, the number of block and the parameteter of regularization are all chosen by MOM cross-validation (see~\cite{lecue2017robust} for more details on the MOM cross-validation procedure) 
Subfigure (a) shows convergence of the error for both algorithms in the first framework. 
Similar performances are observed for both algorithms but Algorithm~\ref{algo_mom} converges faster than Algorithm~\ref{algo_erm}. It may be because the computation of the gradient on a smaller batch of data in step \textbf{5} and \textbf{6} of Algorithm~\ref{algo_mom}  is faster than the one on the entire database in step \textbf{2} of Algorithm~\ref{algo_erm} and that the choice of the median blocks at each descent/ascent step is particularly good in Algorithm~\ref{algo_mom}. Subfigure (b) shows the results in the second framework. The convergence for the alternating gradient ascent/descent algorithm is a bit faster than the one from Algorithm~\ref{algo_erm}, but the performances are the same. Subfigure (c) shows results in the third setup where $\zeta$ is Gaussian and the feature vector $X = (x_1,\cdots,x_p)$ is  heavy-tailed, i.e. $x_1, \ldots, x_p$ are i.i.d. with $x_1 \sim \mathcal T (2)$ -- a Student with degree $2$. Minmax MOM estimators perform better than RERM. It highlights the fact that minmax MOM estimators have optimal subgaussian performance even without the sub-gaussian assumption on the design while RERM are expected to have downgraded statistical properties in heavy-tailed scenariis. Subfigure (d) shows result in the fourth setup where an increasing number of outliers is added in the dataset. Outliers are $X = (10^5)_{1}^p$ and $Y_i=1$ a.s..  While RERM has deteriorated performance just after one outliers was added to the dataset,  minmax MOM estimators maintains good performances up to $10\%$ of outliers. 

\begin{figure}[h!]
	\centering
	\begin{subfigure}[b]{0.3\linewidth} \label{log_no_outliers}
		\includegraphics[width=\linewidth]{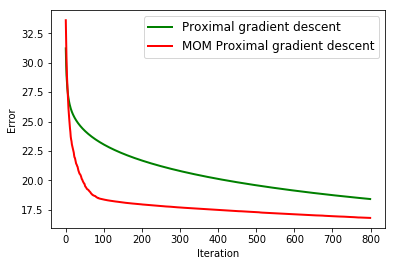}
		\caption{Gaussian design and Gaussian noise.}
	\end{subfigure}
	\begin{subfigure}[b]{0.3\linewidth} \label{log_heavy}
		\includegraphics[width=\linewidth]{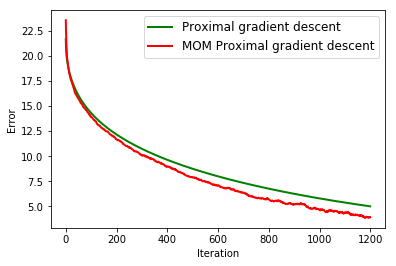}
		\caption{Heavy-tailed noise $\zeta$ (Student distribution of order $2$) and standard Gaussian design.}
	\end{subfigure}
	\begin{subfigure}[b]{0.3\linewidth} \label{log_design}
		\includegraphics[width=\linewidth]{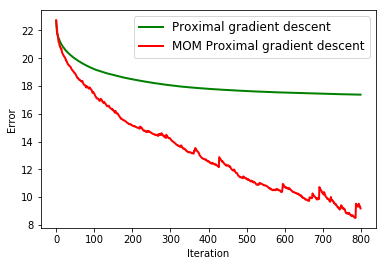}
		\caption{Gaussian noise and  heavy-tailed design (Student distribution of order $2$).}
	\end{subfigure}
	\begin{subfigure}[b]{0.7\linewidth} \label{log_outliers}
		\includegraphics[width=\linewidth]{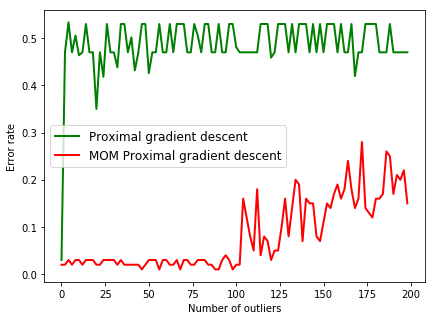}
		\caption{Student of order $2$ design and noise corrupted by outliers.}
	\end{subfigure}
	\caption{$\ell_2$ estimation error rates of RERM and minmax MOM proximal descent algorithms (for the logistic loss and the $\ell_1$ regularization norm) versus time in (a), (b) and (c) and versus number of outliers in (d) in the classification model~\eqref{classif} for $N = 1000$, $p = 400$ and $s = 30$.}
	\label{fig:log}
\end{figure}

\subsection{Huber regression with a Group Lasso penalty }
Let $\ell$ denote the Huber loss function $t\in\bR^d\to \ell_t(x,y)= (y-\inr{x,t})^2/2$ if $|y-\inr{x,t}|\leq \delta$ and $\ell_t(x,y)= \delta|y-\inr{x,t}| - \delta^2/2$ other wise for all $x\in\bR^p$ and $y\in \cY = \bR$. Let $G_1,\cdots,G_M$ be a partition of $\{1,\cdots,p\}$, $\|t\| = \|t\|_{\text{GL}} = \sum_{k=1}^M \|t_{G_k}\|_2$. Figure~\ref{fig:log} presents the results of our simulation for $N=1000$, $p = 400$ for $30$ blocks with a block-sparsity parameter $s=5$. In subfigures (a), (b) and (c), the error is the $L_2$-error between the output of the algorithm and the oracle $t^*$ -- which corresponds here to a $\ell_2^p$ estimation error, given that the design in all cases is isotropic. In subfigure (d) the prediction error on a (non-corrupted) test set of both the RERM and the minmax MOM estimators are depicted. 

The conclusion are the same as for the Lasso Logistic regression: Algorithm~\ref{algo_mom} (regularized minmax MOM) has better performances than algorithm~\ref{algo_erm} (RERM) in case of heavy-tailed inliers and when outliers pollute the dataset while both are robust w.r.t heavy-tailed noise. 

\begin{figure}[h!]
	\centering
	\begin{subfigure}[b]{0.3\linewidth} \label{hub_no_outliers}
		\includegraphics[width=\linewidth]{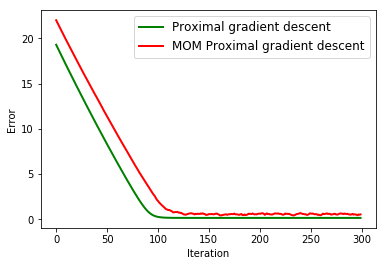}
		\caption{Simulations from model~\eqref{classif} with standard Gaussian design and Gaussian noise}
	\end{subfigure}
	\begin{subfigure}[b]{0.3\linewidth} \label{hub_heavy}
		\includegraphics[width=\linewidth]{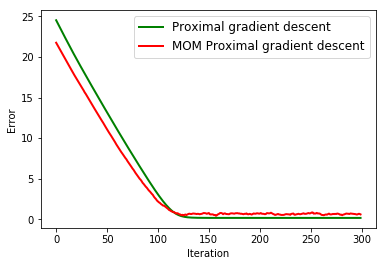}
		\caption{Simulation with heavy-tailed noise $\zeta$ and standard Gaussian design}
	\end{subfigure}
	\begin{subfigure}[b]{0.3\linewidth} \label{hub_design}
		\includegraphics[width=\linewidth]{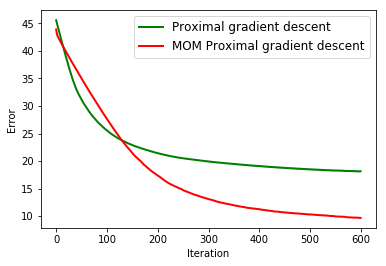}
		\caption{Simulations with Gaussian noise  heavy tailed design (Student distribution)}
	\end{subfigure}
	\begin{subfigure}[b]{0.7\linewidth} \label{hub_outliers}
		\includegraphics[width=\linewidth]{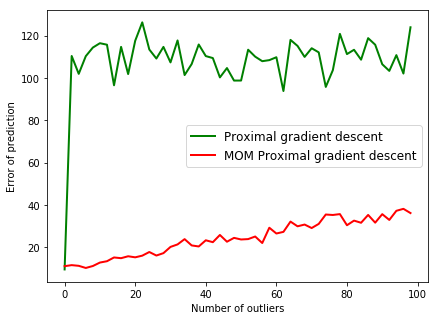}
		\caption{Error of prediction in function of the number of outliers in the dataset}
	\end{subfigure}
	\caption{Results for the Huber regression with Group-Lasso penalization}
	\label{fig:hub}
\end{figure}

\section{Conclusion} 
\label{sec:conclusion}
We obtain estimation and prediction results for RERM and regularized minmax MOM estimators for any Lipschitz and convex loss functions and for any regularization norm. When the norm has some sparsity inducing properties the statistical bounds depend on the dimension  of the low-dimensional structure where the oracle belongs. We develop a systematic way to analyze both estimators by identifying three key idea 1) the local complexity function $r_2$ 2) the sparsity equation 3) the local Bernstein condition. All these quantities and condition depend only on the structure and complexity of a local set around the oracle. This local set is ultimately proved to be the smallest set containing our estimators. We show the versatility of our main meta-theorems on several applications covering two different loss functions and four sparsity inducing regularization norms. Some of them inducing highly structured sparsity concept such as total variation norm. 

On top of these results, we show that the minmax MOM approach is robust to outliers and to heavy-tailed data and that the computation of the key objects such as the complexity functions $r_2$ and a radius $\rho^*$ satisfying the sparsity equation can be done in this corrupted heavy-tailed scenario. Moreover, we show in a simulation section that they can be computed by a simple modification of existing proximal gradient descent algorithms by simply adding a selection step of the central block of data in these algorithms. The resulting algorithms are robust to heavy-tailed data and to few outliers (in both input and output variables) for the examples in Section~\ref{sec:Simu}. 

\acks{We would like to thank Sara van de Geer for pointing to us Lemma 4.3 in~\cite{van2020logistic}. 

Guillaume Lecu{\'e} is supported by a grant overseen by the French National Research Agency (ANR) as part of the“ Investments d’Avenir ”Program (LabEx ECODEC; ANR-11-LABX-0047), by the M{\'e}diam{\'e}trie chair on 'Statistical models and analysis of high-dimensional data' and by  the French ANR PRC grant ADDS (ANR-19-CE48-0005).

Matthieu Lerasle is supported by a grant overseen by the French National Research Agency (ANR) as part of the“ Investments d’Avenir ”Program (LabEx ECODEC; ANR-11-LABX-0047).}

\section{Proof Theorem~\ref{thm:main_subgaussian} }\label{sec:proofs}
All along this section we will write $r(\rho)$ for $r(A,\rho)$. Let $\theta = 1/(3A)$.
The proof is divided into two parts. 
First, we identify an event where the RERM $\hat{f} := \hat{f}^{RERM}_{\lambda}$ is controlled. 
Then, we prove that this event holds with large probability. Let $\rho^*$ satisfying the $A$-sparsity Equation from Definition~\ref{def:SE} and let $\mathcal B = \rho^* B \cap r(\rho^*)B_{L_2}$ and consider
\begin{equation*}
\Omega := \left\{ 
\forall f\in F\cap(f^*+\mathcal B), \quad \big|(P-P_N)\cL_f\big|\leq \theta  r^2(\rho^*) \right\}\enspace.
\end{equation*} 

\begin{Proposition}\label{prop:algebra}Let $\lambda$ be as in \eqref{eq:reg_param_choice}  and let $\rho^*$ satisfy the $A$- sparsity from Definition~\ref{def:SE}. On $\Omega$, one has 
	\begin{align*}
\|\hat{f}-f^{*}\| \leq \rho^*,\quad
\|\hat{f}-f^{*}\|_{L_2} &\leq r( \rho^*) \mbox{ and }
	P\cL_{\hat f} \leq  A^{-1} r^2(\rho^*)\enspace.
	\end{align*}
\end{Proposition} 

\begin{proof} 
Prove first that $\hat f\in f^*+\mathcal B$. 
Recall that
\begin{equation*}
\forall f\in F,\qquad \cL_f^\lambda = \cL_f + \lambda(\norm{f} - \norm{f^*})\enspace.
\end{equation*}
Since $\hat f$ satisfies $P_N\cL_{\hat f}^\lambda\leqslant 0$, it is sufficient to prove that $P_N\cL_f^\lambda>0$ for all $f\in F \backslash(f^*+\mathcal B)$ to get the result. 
The proof relies on the following homogeneity argument.
If $P_N\cL_{f_0}>0$ on the border of $f^*+ \mathcal B$, then $P_N\cL_f>0$ for all $f\in F\setminus\{f^*+ \mathcal B\}$. 

Let $f\in F\setminus\{f^*+ \mathcal B\}$. 
By convexity of $\{f^*+ \mathcal B\}\cap F$, there exists $f_0\in F$ and $\alpha>1$ such that $f-f^* = \alpha(f_0-f^*)$ and $f_0\in\partial(f^*+ \mathcal B)$ where $\partial(f^*+ \mathcal B)$ denotes the border of $f^*+ \mathcal B$ (see, Figure~\ref{fig:homogen}).

\begin{figure}[h]
\centering
\begin{tikzpicture}[scale=0.3]
\draw (0,0) circle (8cm);
\draw (1.5,0) node {$f^*$};
\filldraw (0,0) circle (0.2cm);
\draw (-10,0) -- (0,10) -- (10,0) -- (0,-10) -- (-10,0);
\draw (11,0) node {$\rho^* B$};
\draw (9.2,5) node {$r(\rho^*) B_{L_2}$};
\draw (-1,-4) node {$\mathbf{\boldsymbol{\rho}^* B\cap r(\boldsymbol{\rho}^*) B_{L_2}}$};
\draw (-10,8) -- (0,0);
\draw (-11,8) node {$f$};
\filldraw (-10,8) circle (0.2cm);
\draw (-5.5,3.5) node {$f_0$};
\filldraw (-5.6,4.5) circle (0.2cm);

\draw[ultra thick] (-7.7,2.3) -- (-2.3,7.7) arc (110:70:6.8) -- (7.7,2.3) arc (20:-20:6.8) -- (2.3,-7.7) arc (-70:-110:6.8) -- (-7.7, -2.3) arc (-160:-200:6.8);

\end{tikzpicture}
\caption{Construction of $f_0$.}
\label{fig:homogen}
\end{figure}
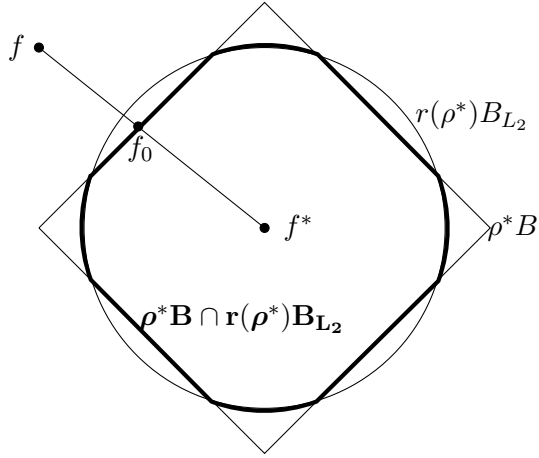

For all $i \in \{1,\cdots,N \}$, let $\psi_i: \mathbb R \rightarrow \mathbb R $ be the random function defined for all $u\in \R$ by 
\begin{equation}\label{eq:fct_psi}
\psi_i(u) = \ell (u + f^{*}(X_i), Y_i) - \ell (f^{*}(X_i), Y_i)\enspace.
\end{equation}
By construction, for any $i$, $\psi_i(0) = 0$ and $\psi_i$ is convex because $\ell$ is. Hence, $\alpha\psi_i(u) \leq \psi_i(\alpha u)$ for all $u\in\mathbb R$ and $\alpha \geq 1$. In addition, $\psi_i(f(X_i) - f^{*}(X_i) )=  \ell (f(X_i), Y_i) - \ell (f^{*}(X_i), Y_i) $. Therefore,
\begin{align} \label{conv_arg}
\nonumber P_N \cL_f & = \frac{1}{N} \sum_{i=1}^{N}  \psi_i \big( f(X_i)- f^{*}(X_i) \big )= \frac{1}{N} \sum_{i=1}^{N} \psi_i \big( \alpha( f_0(X_i)- f^{*}(X_i) ) \big)\\
&\geq \frac{\alpha}{N} \sum_{i=1}^{N}   \psi_i( f_0(X_i)- f^{*}(X_i)) = \alpha P_N \cL_{f_0}\enspace.
\end{align}
For the regularization term, by the triangular inequality, 
\begin{align*}
\norm{f}-\norm{f^*} & = \norm{f^*+\alpha(f_0-f^*)}-\norm{f^*} \geq \alpha(\norm{f_0}-\norm{f^*})\enspace.
\end{align*}From the latter inequality, together with \eqref{conv_arg}, it follows that
\begin{equation}\label{eq:homog_final}
P_N\cL_f^\lambda \geq \alpha P_N\cL_{f_0}^{\lambda}\enspace.
\end{equation}
As a consequence, if $P_N\cL_{f_0}^{\lambda}>0$ for all $f_0\in F \cap \partial(f^*+ \mathcal B)$ then  $P_N\cL_f^\lambda>0$ for all $f\in F \backslash (f^*+ \mathcal B)$.\\

In the remaining of the proof, assume that $\Omega$ holds and let $f_0\in F \cap \partial(f^*+ \mathcal B)$. As $f_0\in F\cap (f^*+ \mathcal B)$, on $\Omega$,
\begin{equation}\label{eq:use_event_omega}
|(P-P_N)\cL_{f_0}|\leq \theta r^2(\rho^*)\enspace.
\end{equation}  
By definition of $\mathcal B$, as $f_0\in \partial(f^*+ \mathcal B)$, either: 1) $\norm{f_0-f^*}=\rho^*$ and $\norm{f_0-f^*}_{L_2}\leq r(\rho^*)$ so $\alpha = \norm{f-f^*}/\rho^*$ or 2) $\norm{f_0-f^*}_{L_2}=r(\rho^*)$ and $\norm{f_0-f^*}\leq \rho^*$ so $\alpha = \norm{f-f^*}_{L_2}/r(\rho^*)$. We treat these cases independently.  

Assume first that $\norm{f_0-f^*}=\rho^*$ and $\norm{f_0-f^*}_{L_2}\leq r(\rho^*)$. Let $v\in E$ be such that $\norm{f^*-v}\leq \rho^*/20$ and $g\in \partial \norm{\cdot}(v)$. We have
\begin{align*}
&\norm{f_0}-\norm{f^*} \geq \norm{f_0} - \norm{v} - \norm{f^*-v}\geq \inr{g,f_0-v} - \norm{f^*-v}\\
&\geqslant\inr{g,f_0-f^*} - 2\norm{f^*-v}\geqslant\inr{g,f_0-f^*} - \rho^*/10\enspace.
\end{align*}
As the latter result holds for all $v\in f^*+(\rho^*/20)B$ and $g\in \partial \norm{\cdot}(v)$, since $f_0-f^*\in \rho^* S\cap r(\rho^*)B_{L_2}$, it yields 
\begin{equation}\label{eq:homogen_reg}
\norm{f_0}-\norm{f^*}\geq \Delta(\rho^*) - \rho^*/10\geq 7\rho^*/10\enspace.
\end{equation}
Here, the last inequality holds because $\rho^*$ satisfies the sparsity equation. Hence,
\begin{equation}\label{eq:main1}
P_N \cL_{f}^\lambda = P_N \cL_{f}+ \lambda\left(\norm{f} - \norm{f^*}\right)\geq \alpha (P_N\cL_{f_0} + 7\lambda\rho^*/10)\enspace.
\end{equation}
Thus, on $\Omega$, since $\lambda > 10\theta r^2(\rho^{*}) ^{2}/(7 \rho^{*})$,
\begin{equation*}
P_N\cL_{f_0} + 7\lambda \rho^*/10 = P\cL_{f_0} + (P_N-P)\cL_{f_0} + 7 \lambda \rho^*/10\geq -\theta r^2(\rho^*) + 7\lambda \rho^*/10>0\enspace.
\end{equation*}

Assume now that $\norm{f_0-f^*}_{L_2}=r(\rho^*)$ and $\norm{f_0-f^*}\leq \rho^*$. 
By Assumption~\ref{assum:fast_rates}, on $\Omega$,
\begin{align*}
P_N \cL_{f}^\lambda&\geqslant P_N \cL_{f_0} - \lambda\norm{f_0-f^*} \geqslant P\cL_{f_0} + (P_N-P)\cL_{f_0} - \lambda\rho^*\\
&\geqslant A^{-1}\norm{f_0-f^*}_{L_2}^{2}-\theta r^2(\rho^*) - \lambda\rho^* \geqslant (A^{-1} - \theta)r^2(\rho^*) - \lambda\rho^*\enspace.
\end{align*}
From \eqref{eq:reg_param_choice}, $\lambda < (A^{-1}-\theta)   r^2(\rho^{*}) ^{2}/\rho^{*}$, thus $P_N \cL_{f}^\lambda>0$. Together with \eqref{eq:main1}, this proves that $\hat{f}\in f^*+\mathcal B$. 
Now, on $\Omega$, this implies that $|(P-P_N)\cL_{\hat f}|\leq \theta r^2(\rho^*)$, so by definition of $\hat{f}$,
\begin{equation*}
P\cL_{\hat f} = P_N\cL_{\hat f}^\lambda + (P-P_N)\cL_{\hat f}+\lambda(\|f^*\|-\|\hat{f}\|)\leqslant \theta r^2(\rho^*)+\lambda\rho^*\leqslant A^{-1}r^2(\rho^*)\enspace.
\end{equation*}
\end{proof}

To prove that $\Omega$ holds with large probability, the following result from \cite{pierre2017estimation} is useful. 
\begin{Lemma} \cite[Lemma~9.1]{pierre2017estimation} \label{lem:subgauss}
Grant Assumptions~\ref{assum:lip} and~\ref{ass:sub-gauss}. Let $F^\prime\subset F$ denote a subset with finite $L_2$-diameter $d_{L_2}(F^\prime)$. For every $u>0$, with probability at least $1-2\exp(-u^2)$
	\begin{equation*}
	\sup_{f,g\in F^\prime}\left|(P-P_N)(\cL_f-\cL_g)\right|\leq \frac{16LL_0}{\sqrt{N}} \left(w(F^\prime) +   u d_{L_2}(F^\prime)\right)\enspace.
	\end{equation*} 
\end{Lemma}
It follows from Lemma~\ref{lem:subgauss} that for any $u>0$, with probability larger that $1-2\exp(-u^2)$,
\begin{align*}
\sup_{f \in  F \cap (f^{*} + \mathcal{B}) } \big | (P-P_N)\cL_f  \big|  &\leqslant \sup_{f,g \in F\cap(f^{*} +  \mathcal{B}) } \big | (P-P_N)(\cL_f-\cL_g)  \big| \\
& \leqslant \frac{16LL_0}{\sqrt{N}} \big(   w(F \cap(f^* +   \mathcal{B}) ) + ud_{L_2} (F\cap(f^* +  \mathcal{B}) )  \big)\enspace.
\end{align*}
It is clear that $d_{L_2} (F\cap (f^*  +  \mathcal{B}) ) \leqslant r(\rho^{*})$. By definition of the complexity function~\eqref{def:function_r}, for $u = \theta \sqrt{N} r(\rho^{*} )/(32LL_0) $, we have with probability at least $1-2\exp\big(-\theta^2N r^2(\rho^{*} ) /(32LL_0 )^2 \big)$,
\begin{gather*}
\forall f\in F\cap(f^*+ \mathcal{B} ),\qquad \big | (P-P_N) \cL_f \big|  \leq \theta  r^2(\rho^{*} )\enspace.
\end{gather*}

\section{Proof Theorem~\ref{main:theo}} \label{proof_mom}
All along the proof, the following notations will be used repeatedly.
\begin{gather*}
 \theta = \frac{1}{34A}, \quad \gamma = \theta /( 192 L ) \quad \hat{f} = \hat f_{K,\lambda}  \enspace.
\end{gather*}

The proof is divided into two parts. First, we identify an event where the minmax MOM estimator $\hat{f}$ is controlled. 
Then, we prove that this event holds with large probability. 
Let $K \geqslant 7|\cO| /3$,  and $\kappa \in \{1,2\}$ let
\begin{equation*}
C_{K,r,\kappa} = \max \bigg(\frac{96L^2K}{\theta^2 N},r_2^2(\gamma,\kappa \rho^{*}) \bigg) \quad \mbox{and} \quad \lambda = 10\theta \frac{C_{K,r,2}}{\rho^{*} } \enspace.
\end{equation*}
Let $\mathcal{B}_\kappa=\sqrt{C_{K,r,\kappa}}B_{L_2}\cap \kappa \rho^{*}B$. Consider the following event
\begin{equation} \label{omega_mom}
 \Omega_K   =  \bigg \{ \forall \kappa\in\{1,2\},\ \forall f \in F\cap f^*+\mathcal{B}_\kappa,  \;
\sum_{k=1}^KI\bigg(\bigg| (P_{B_k} - P)(\ell_f-\ell_{f^{*}}) \bigg| \leq \theta C_{K,r,\kappa}\bigg)\geqslant \frac K2 \bigg \}
\end{equation}

\subsection{Deterministic argument}

\begin{Lemma}\label{lem:Basic}
$\hat f - f^{*} \in\mathcal{B}_\kappa$ if there exists $\eta>0$ such that
	\begin{align}\label{obj_proof} 
	& \sup_{f \in f^{*} +  F \backslash \mathcal{B}_\kappa } \quad \MOM{K}{\ell_{f^{*}}-\ell_f}  + \lambda \big( \|f^{*}\|-\|f\|  \big)< - \eta \enspace,\\
	& \sup_{f \in F} \MOM{K}{\ell_{f^{*}}-\ell_f} + \lambda \big( \|f^{*}\|-\|f\|  \big) \leq  \eta  \enspace.\label{cond_gen}
	\end{align}

\end{Lemma}
\begin{proof}
	For any $f\in F$, denote by $S(f)=\sup_{g\in F}\text{MOM}_K[\ell_f-\ell_g] + \lambda \big( \|f\|-\|g\|  \big)$.
	If \eqref{obj_proof} holds, by homogeneity of $\text{MOM}_K$, any $f\in f^{*} + F\backslash \mathcal{B}_\kappa$ satisfies
	\begin{equation}
	S(f)\geqslant \inf_{f\in f^{*} + F \backslash \mathcal{B}_\kappa} \text{MOM}_K[\ell_f-\ell_{f^*}]+ \lambda \big( \|f\|-\|f^{*}\|  \big)> \eta\enspace.\label{eq:Task2} 
	\end{equation}
On the other hand, if \eqref{cond_gen} holds,
	\begin{align} \nonumber
	S(f^{*})= & \sup_{f\in F }\text{MOM}_K[\ell_{f^*}-\ell_f] + \lambda \big( \|f^{*}\|-\|f\|  \big)  \leqslant \eta\enspace.\label{eq:Task1}
	\end{align}
Thus, by definition of $\hat{f}$ and \eqref{cond_gen},
	\[
	S(\hat{f})\leqslant S(f^*)\leqslant \eta\enspace.
	\]
Therefore, if \eqref{obj_proof} and \eqref{cond_gen} hold,
	$\hat{f} \in f^{*} + \mathcal{B}_\kappa$.	
\end{proof}  
It remains to show that, on $\Omega_K$, Equations~\eqref{obj_proof} and~\eqref{cond_gen} hold for $\kappa = 2$.\\

Let $\kappa\in\{1,2\}$ and $f \in  F\cap \mathcal{B}_\kappa $. On $\Omega_K$, there exist more than $K/2$ blocks $B_k$ such that  
\begin{equation} \label{mom_F1}
\bigg| (P_{B_k} - P)(\ell_f-\ell_{f^{*}}) \bigg| \leq \theta C_{K,r,\kappa} \enspace.
\end{equation}
It follows  that
\begin{align*}
\sup_{f \in f^{*} +  F\cap \mathcal{B}_\kappa }  \MOM{K}{\ell_{f^{*}}-\ell_f} \leq \theta C_{K,r,\kappa }
\end{align*}
In addition,  $\|f\|-\|f^{*}\| \leq \kappa \rho^{*}$. Therefore, from the choice of $\lambda$, on $\Omega_K$,  one has
\begin{equation} \label{control_F1}
\sup_{f \in f^* + F \cap \mathcal{B}_\kappa}  \MOM{K}{\ell_{f^{*}}-\ell_f} + \lambda \big( \|f^{*}\|-\|f\| \big) \leq  (1+10\kappa)\theta C_{K,r,\kappa}  \enspace.
\end{equation}

Assume that $f$ belongs to $ F \backslash  \mathcal{B}_\kappa$. By convexity of $F$, there exists $f_0 \in   f^{*} + F \cap \mathcal{B}_\kappa $ and $\alpha > 1$ such that
\begin{equation} \label{homo_argument_MOM}
f = f^{*} + \alpha (f_0 - f^{*}) \enspace. 
\end{equation}
For all $i \in \{1,\cdots,N \}$, let $\psi_i: \mathbb R \rightarrow \mathbb R $ be the random function defined for all $u\in \R$ by 
\begin{equation}\label{eq:fct_psi_MOM}
\psi_i(u) = \ell (u + f^{*}(X_i), Y_i) - \ell (f^{*}(X_i), Y_i)\enspace.
\end{equation}
The functions $\psi_i$ are convex and satisfy $\psi_i(0) = 0$. Thus $\alpha\psi_i(u) \leq \psi_i(\alpha u)$ for all $u\in\mathbb R$ and $\alpha > 1$ and $\psi_i(f(X_i) - f^{*}(X_i) )=  \ell (f(X_i), Y_i) - \ell(f^{*}(X_i), Y_i) $. Hence, for any block $B_k$,  
\begin{align} \label{conv_arg_MOM}
\nonumber P_{B_k} \cL_f & = \frac{1 }{|B_k|} \sum_{i \in B_k} \psi_i \big( f(X_i)- f^{*}(X_i) \big) = \frac{1}{|B_k|} \sum_{i \in B_k}   \psi_i \big(\alpha( f_0(X_i)- f^{*}(X_i) )\big)\\
&\geq \frac{\alpha}{|B_k|} \sum_{i \in B_k}   \psi_i \big( f_0(X_i)- f^{*}(X_i) \big) = \alpha P_{B_k} \cL_{f_0}\enspace.
\end{align}
By the triangular inequality,
\begin{align*}
\norm{f}-\norm{f^*} & = \norm{f^*+\alpha(f_0-f^*)}-\norm{f^*} \geq \alpha(\norm{f_0}-\norm{f^*}).
\end{align*}
Together with \eqref{conv_arg_MOM}, this yields, for all block $B_k$
\begin{equation}\label{eq:homog_final_MOM}
P_{B_k}\cL_f^\lambda \geq \alpha P_{B_k}\cL_{f_0}^{\lambda}\enspace.
\end{equation}
As $f_0 \in F \cap \mathcal{B}_\kappa$, on $\Omega_K$, 
\begin{equation}\label{eq:use_event_omega_MOM}
|(P-P_{B_k})\cL_{f_0}|\leq \theta C_{K,r,\kappa}.
\end{equation}  

As $f_0$ can be chosen in $\partial (f^*+\mathcal{B}_\kappa)$, either: 1) $\norm{f_0-f^*}= \kappa\rho^*$ and $\norm{f_0-f^*}_{L_2}\leq \sqrt{C_{K,r,\kappa}}$ or 2) $\norm{f_0-f^*}_{L_2}= \sqrt{C_{K,r,\kappa}}$ and $\norm{f_0-f^*}\leq \kappa \rho^*$.

Assume first that $\norm{f_0-f^*}= \kappa\rho^*$ and $\norm{f_0-f^*}_{L_2}\leq \sqrt{C_{K,r,\kappa}}$. Since the sparsity equation is satisfied for $\rho = \rho^*$, it is also satisfied for $\kappa \rho^{*}$. By \eqref{eq:homogen_reg},
\begin{equation}\label{eq:homogen_reg_MOM}
\lambda \big( \norm{f_0}-\norm{f^*} \big) \geq 7\lambda \kappa \rho^*/10 = 7\kappa C_{K,r,2} \enspace.
\end{equation}
Therefore, on $\Omega_K$, there are more than $K/2$ blocks $B_k$ where 
\begin{equation}
P_{B_k} \cL_f^\lambda \geq \alpha P_{B_k}\cL_{f_0}^{\lambda} \geq \alpha \bigg( -\theta C_{K,r,\kappa}  + \frac{7\kappa \lambda \rho^{*}}{10} \bigg)  \geqslant \alpha (7\kappa-1) \theta C_{K,r,2} \enspace.
\end{equation}
It follows that 
\begin{equation} \label{homo_1}
\MOM{K}{\ell_f - \ell_{f^*}} + \lambda \big( \|f\|-\|f^{*}\| \big) \geqslant \alpha \theta \big(   7\kappa C_{K,r,2} - C_{K,r,\kappa}  \big)
C_{K,r,2} \enspace.
\end{equation}

Assume that $\norm{f_0-f^*}_{L_2}= \sqrt{C_{K,r,\kappa}}$ and $\norm{f_0-f^*}\leq \kappa  \rho^*$. By Assumption~\ref{assum:fast_rates_MOM}, on $\Omega_K$, there exist more than $K/2$ blocks $B_k$ where
\begin{align*}
P_{B_k} \cL_f^\lambda&\geqslant P_{B_k} \cL_{f_0} - \lambda\norm{f_0-f^*} \geq P\cL_{f_0} + (P_{B_k}-P)\cL_{f_0} - \lambda \kappa \rho^*\\
&\geq A^{-1}\norm{f_0-f^*}_{L_2}^{2}-\theta C_{K,r,\kappa} - \kappa\lambda \rho^* = \theta(33 C_{K,r,\kappa}-10\kappa C_{K,r,2})\enspace.
\end{align*}
It follows that 
\begin{equation} \label{homo_2}
\MOM{K}{\ell_f- \ell_{f^*}} + \lambda \big( \|f\|-\|f^*\| \big) \geqslant  \alpha\theta(33 C_{K,r,\kappa}-10\kappa C_{K,r,2}) \enspace.
\end{equation}

From Equations~\eqref{control_F1},~\eqref{homo_1} and~\eqref{homo_2} with $\kappa = 1 $, it follows that
\begin{equation}\label{eq:supMOM}
\sup_{f \in F} \MOM{K}{\ell_{f^{*}}-\ell_f} + \lambda \big( \|f^{*}\|-\|f\| \big) \leq 11\theta C_{K,r,2}\enspace.
\end{equation}
Therefore,~\eqref{cond_gen} holds with $\eta = 11\theta C_{K,r,2}$. 
Now, Equations~\eqref{homo_1} and~\eqref{homo_2} with $\kappa = 2$ yield
\begin{align*}
\sup_{f \in f^{*} +  F \backslash \mathcal{B}_2 } \MOM{K}{\ell_{f^{*}}-\ell_f} + \lambda \big( \|f^{*}\|-\|f\| \big) & \leqslant-13\alpha\theta C_{K,r,2}  < -11\theta C_{K,r,2} \enspace.
\end{align*}
Therefore, Equation~\eqref{obj_proof} holds with $\eta = 11\theta C_{K,r,2}$. 
Overall, Lemma~\ref{lem:Basic} shows that $\hat{f}\in\mathcal{B}_2$.
On $\Omega_K$, this implies that there exist more than $K/2$ blocks $B_k$ where $P \cL_{\hat{f}} \leq P_{B_k} \cL_{\hat{f}} + \theta C_{K,r,2}$. In addition, by definition of $\hat f$ and \eqref{eq:supMOM}, 
\begin{equation*}
\MOM{K}{\ell_{\hat f} - \ell_{f^{*}}} + \lambda (\|\hat{f}\| - \|f^{*}\| ) \leqslant \sup_{f \in F} \MOM{K}{\ell_{f^{*}} - \ell_{f}} + \lambda (\|f^{*}\| - \|f\| ) \leq 11\theta C_{K,r,2}\enspace.
\end{equation*}
This means that there exist at least $K/2$ blocks $B_k$ where $P_{B_k} \cL_{\hat{f}} + \lambda(\|\hat{f}\| - \|f^*\|) \leqslant 11\theta C_{K,r,2}$.
As $\|\hat{f}\| - \|f^*\|\geqslant -\|\hat{f} - f^*\| \geqslant -2\rho^*$, on these blocks, $P_{B_k}\cL_{\hat f} \leq 31\theta C_{K,r,2}$. Therefore, there exists at least one block $B_k$ for which simultaneously 
$P \cL_{\hat{f}} \leq P_{B_k} \cL_{\hat{f}} + \theta C_{K,r,2}$ and $P_{B_k}\cL_{\hat f} \leq 31\theta C_{K,r,2}$. This shows that
 $P\cL_{\hat{f}} \leq 32\theta C_{K,r,2} \leq A^{-1} C_{K,r,2}$.

\subsection{Control of the stochastic event} \label{sec:mom}

\begin{Proposition}\label{lem:close}
	Grant Assumptions~\ref{assum:lip},~\ref{assum:convex},~\ref{assum:moments} and~\ref{assum:fast_rates_MOM}. Let $K \geq 7|\cO|/3$. Then $\Omega_K$ holds with probability larger than $1-2\exp(- K/504)$.
\end{Proposition}

\begin{proof}
	Let $\mathcal F = F \cap \left(f^{*} + \mathcal{B}_{\kappa}\right)$  and let $\phi(t) = \mathbb{1} \{ t\geq 2 \} + (t-1) \mathbb{1} \{1 \leq t \leq 2 \}$. This function satisfies $\forall t \in \mathbb{R}^+ \hspace{0.4cm} \mathbb{1} \{ t\geq 2 \} \leq \phi(t) \leq \mathbb{1} \{ t\geq 1 \}$. Let $W_k = ((X_i,Y_i))_{i \in B_k}$ and, for any $f\in\cF$, let $G_f(W_k) = (P_{B_k} - P)(\ell_f-\ell_{f^{*}})$. Let also $C_{K,r,\kappa}=\max \bigg( 96L^2K/(\theta^2 N), r_2^2(\gamma,\kappa\rho^{*}) \bigg) $. For any $f\in\cF$, let
	\begin{align*}
	z(f) & = \sum_{k =1}^K \mathbb{1} \{|G_f(W_k)|\leq \theta C_{K,r,\kappa}\}\enspace.
	\end{align*}
	Proposition \ref{lem:close} will be proved if $z(f)\geqslant K/2$ with probability larger than $1-e^{-K/504}$.
	Let $\mathcal{K}$ denote the set of indices of blocks which have not been corrupted by outliers, $\mathcal{K} = \{k \in \{1,\cdots,K \} : B_k \subset \mathcal{I}\}$, where we recall that $\cI$ is the set of informative data. Basic algebraic manipulations show that
	\begin{multline}\label{eq:LBzf1}
	z(f) \geqslant |\mathcal{K}| - \sup_{f\in \mathcal F} \sum_{k \in \mathcal{K}} \bigg( \phi(2(\theta C_{K,r,\kappa})^{-1} | G_f(W_k)|) - \mathbb{E} \phi(2(\theta C_{K,r,\kappa})^{-1} | G_f(W_k)|) \bigg) \\
	-  \sum_{k \in \mathcal{K} } \mathbb{E}\phi(2(\theta C_{K,r,\kappa})^{-1} | G_f(W_k)|) \enspace.
	\end{multline} 
	The last term in \eqref{eq:LBzf1} can be bounded from below as follows. Let $f\in \cF$ and $k\in \cK$,
	\begin{align*}
	\mathbb{E}\phi(2(\theta C_{K,r,\kappa})^{-1} | G_f(W_k)|) & \leqslant \mathbb{P} \bigg( |G_f(W_k)| \geq \frac{\theta C_{K,r,\kappa}}{2} \bigg) \leqslant \frac{4\mathbb{E}G_f(W_k)^2}{(\theta C_{K,r,\kappa})^2}   \\
	&  \leqslant  \frac{4K^2}{\theta^2 C_{K,r,\kappa}^2N^2} \sum_{i \in B_k} \mathbb{E} [(\ell_f-\ell_{f^{*}})^2(X_i,Y_i)] \leq \frac{4L^2K}{\theta^2 C_{K,r,\kappa}^2N}\|f-f^{*}\|^2_{L_2} \enspace.
	\end{align*}
	The last inequality follows from Assumption \ref{assum:moments}. 
	Since $\|f-f^{*}\|_{L_2} \leq \sqrt{C_{K,r,\kappa}}$,
	\[
	\mathbb{E}\phi(2(\theta C_{K,r,\kappa})^{-1} | G_f(W_k)|)\leqslant \frac{4L^2K}{\theta^2 C_{K,r,\kappa}N}\enspace.
	\]
	As $C_{K,r,\kappa}\geqslant 96L^2K/(\theta^2 N)$,  
	\begin{equation*}
	\mathbb{E}\phi(2(\theta C_{K,r,\kappa})^{-1} | G_f(W_k)|)  \leq \frac{1}{24} \enspace.
	\end{equation*}
	Plugging this inequality in \eqref{eq:LBzf1} yields
	\begin{align}\label{res::1}
	z(f) \geq |\mathcal{K}|(1-\frac1{24}) -\sup_{ f \in \mathcal F
	} \sum_{k \in \mathcal{K}} \bigg( \phi(2(\theta C_{K,r,\kappa})^{-1} | G_f(W_k)|) - \mathbb{E} \phi(2(\theta C_{K,r,\kappa})^{-1} | G_f(W_k)|) \bigg) \enspace.
	\end{align}
	Using the Mc Diarmid's inequality, with probability larger than $1-\exp(- |\cK|/288 )$,
	we get 
	\begin{align*}
	\sup_{f \in \mathcal F} & \sum_{k \in \mathcal{K}} \bigg( \phi(2(\theta C_{K,r,\kappa})^{-1} | G_f(W_k)|)  - \mathbb{E} \phi(2(\theta C_{K,r,\kappa})^{-1} | G_f(W_k)|) \bigg)   \\
	& \leq  \frac{|\mathcal{K}|}{24} + \mathbb{E} \sup_{f \in \mathcal F}   \sum_{k \in \mathcal{K}} \bigg( \phi(2(\theta C_{K,r,\kappa})^{-1} | G_f(W_k)|) - \mathbb{E}  \phi(2(\theta C_{K,r,\kappa})^{-1} | G_f(W_k)|) \bigg)\enspace.
	\end{align*}
By the symmetrization lemma, it follows that, with probability larger than $1-\exp(- |\cK|/288 )$,

	\begin{multline*}
	\sup_{f \in \mathcal F}  \sum_{k \in \mathcal{K}} \bigg( \phi(2(\theta C_{K,r,\kappa})^{-1} | G_f(W_k)|)  - \mathbb{E} \phi(2(\theta C_{K,r,\kappa})^{-1} | G_f(W_k)|) \bigg) \\
	\leqslant \frac{|\mathcal{K}|}{24} + 2 \mathbb{E} \sup_{f \in \mathcal F }   \sum_{k \in \mathcal{K}} \sigma_k \phi(2(\theta C_{K,r,\kappa})^{-1} | G_f(W_k)|) \enspace.
	\end{multline*}
As $\phi$ is 1-Lipschitz with $\phi(0)=0$, the contraction lemma from \cite{ledoux2013probability}and yields 
	\begin{align*}
	\sup_{f \in \mathcal F} \sum_{k \in \cK} \bigg( \phi(2(\theta C_{K,r,\kappa})^{-1} | G_f(W_k)|) -  &\mathbb{E} \phi(2(\theta C_{K,r,\kappa})^{-1} | G_f(W_k)|) \bigg)  \\
	& \leqslant \frac{|\mathcal{K}|}{24} + \frac{4}{\theta} \mathbb{E}  \sup_{f \in \mathcal F}  \quad  \sum_{k \in \mathcal{K}}  \sigma_k \frac{ G_f(W_k)}{C_{K,r,\kappa}}  \\
	& = \frac{|\mathcal{K}|}{24} + \frac{4}{\theta} \mathbb{E} \sup_{f \in \mathcal F} \quad \sum_{k \in \mathcal{K}}  \sigma_k \frac{(P_{B_k}- P)(\ell_f-\ell_{f^{*}})}{C_{K,r,\kappa}}  \\
	\end{align*}
	For any $k\in \cK$, let $(\sigma_i)_{i \in B_k}$ independent from $(\sigma_k)_{k \in \mathcal{K}}$, $(X_i)_{i \in \cI}$ and $(Y_i)_{i \in \cI}$. The vectors  $(\sigma_i \sigma_k (\ell_f-\ell_{f^{*}})(X_i,Y_i))_{i,f}$ and $(\sigma_i (\ell_f-\ell_{f^{*}})(X_i,Y_i))_{i,f}$ have the same distribution. Thus, by the symmetrization and contraction lemmas, with probability larger than $1-\exp(- |\cK|/288)$,
	\begin{align}
\notag	\sup_{f \in \mathcal F } \sum_{k \in \mathcal{K}} \bigg( \phi(2C_{K,r,\kappa}^{-1} | G_f(W_k)|)& -  \mathbb{E} \phi(2C_{K,r,\kappa}^{-1} | G_f(W_k)|) \bigg)  \\
\notag	&  \leq \frac{|\mathcal{K}|}{24}+ \frac{8}{\theta} \mathbb{E}  \sup_{f \in \mathcal F}  \quad \sum_{k \in \mathcal{K}} \frac{1}{|B_k|} \sum_{i \in B_k} \sigma_i \frac{ (\ell_f-\ell_{f^{*}})(X_i,Y_i)}{C_{K,r,\kappa}}   \\ 
\notag	& = \frac{|\mathcal{K}|}{24} + \frac{8K}{\theta N} \mathbb{E} \sup_{f \in \mathcal F} \quad \sum_{i \in \cup_{k \in \mathcal{K}} B_k}  \sigma_i  \frac{(\ell_f-\ell_{f^{*}})(X_i,Y_i)}{C_{K,r,\kappa}} \\
\label{eq:ConcB}	& \leq \frac{|\mathcal{K}|}{24} + \frac{8LK}{\theta N} \mathbb{E} \sup_{f \in \mathcal F} \bigg| \sum_{i \in \cup_{k \in \mathcal{K}} B_k}  \sigma_i \frac{(f-f^{*})(X_i)}{C_{K,r,\kappa}} \bigg|\enspace.
	\end{align}
	Now either 1) $K \leq  \theta^2 r_2^2(\gamma,\kappa\rho^{*})N /(96L^2)$ or 2) $K >  \theta^2 r_2^2(\gamma,\kappa\rho^{*})N /(96L^2)$. 
	Assume first that $K \leq  \theta^2 r_2^2(\gamma,\kappa\rho^{*})N /(96L^2)$, so $C_{K,r,\kappa} = r_2^2(\gamma,\kappa\rho^{*})$ and by definition of the complexity parameter 
	\begin{align*}
	\mathbb{E} \sup_{f \in \mathcal F } \bigg| & \sum_{i \in \cup_{k \in \mathcal{K}} B_k}  \sigma_i \frac{(f-f^{*})(X_i)}{C_{K,r,\kappa}} \bigg| =  \mathbb{E} \sup_{f \in \mathcal F } \quad \frac{1}{r_2^2(\gamma,\kappa\rho^{*})} \bigg| \sum_{i \in \cup_{k \in \mathcal{K}} B_k}  \sigma_i  (f-f^{*})(X_i)\bigg| \leq \frac{\gamma |\cK| N}{K}\enspace.
	\end{align*}
	If $K >  \theta^2 r_2^2(\gamma,\kappa\rho^{*})N /(96L^2)$, $C_{K,r,\kappa} = 96L^2K/( \theta^2 N)$. Write $\cF=\cF_1\cup\cF_2$, where 
	\[
	\cF_1:=\{f\in \cF:\norm{f-f^*}_{L_2}\leqslant r_2(\gamma,\kappa\rho^{*})\},\qquad \cF_2=\cF\setminus\cF_1\enspace.
	\]
	Then,
	\begin{multline*}
	\mathbb{E} \sup_{f \in \mathcal F} \bigg|  \sum_{i \in \cup_{k \in \mathcal{K}} B_k}  \sigma_i \frac{(f-f^{*})(X_i)}{C_{K,r,\kappa}} \bigg|\\
	 =\frac1{C_{K,r,\kappa}}\E \bigg[\sup_{f \in \mathcal F_1} \bigg|  \sum_{i \in \cup_{k \in \mathcal{K}} B_k}  \sigma_i(f-f^{*})(X_i) \bigg|\vee \sup_{f \in \mathcal F_2} \bigg|  \sum_{i \in \cup_{k \in \mathcal{K}} B_k}  \sigma_i(f-f^{*})(X_i) \bigg| \bigg]\enspace.
	\end{multline*}
For any $f\in \cF_2$, $g=f^*+(f-f^*)r_2(\gamma,\kappa\rho^{*})/\sqrt{C_{K,r,\kappa}}\in \cF_1$ and 
\[
\bigg|  \sum_{i \in \cup_{k \in \mathcal{K}} B_k}  \sigma_i(f-f^{*})(X_i) \bigg|=\frac{\sqrt{C_{K,r,\kappa}}}{r_2(\gamma,\kappa\rho^{*})}\bigg|  \sum_{i \in \cup_{k \in \mathcal{K}} B_k}  \sigma_i(g-f^{*})(X_i) \bigg|\enspace.
\]
It follows that 
\[
 \sup_{f \in \mathcal F_2} \bigg|  \sum_{i \in \cup_{k \in \mathcal{K}} B_k}  \sigma_i(f-f^{*})(X_i) \bigg|\leqslant \frac{\sqrt{C_{K,r,\kappa}}}{r_2(\gamma,\kappa\rho^{*})}\sup_{f \in \mathcal F_1} \bigg|  \sum_{i \in \cup_{k \in \mathcal{K}} B_k}  \sigma_i(f-f^{*})(X_i) \bigg|\enspace.
\]
Hence,
\[
\mathbb{E} \sup_{f \in \mathcal F} \bigg|  \sum_{i \in \cup_{k \in \mathcal{K}} B_k}  \sigma_i \frac{(f-f^{*})(X_i)}{C_{K,r,\kappa}} \bigg|\leqslant \frac1{r_2(\gamma,\kappa\rho^*)\sqrt{C_{K,r,\kappa}}}\mathbb{E} \sup_{f \in \mathcal F_1} \bigg|  \sum_{i \in \cup_{k \in \mathcal{K}} B_k}  \sigma_i (f-f^{*})(X_i) \bigg|\enspace.
\]
By definition of $r_2$, this implies
\[
\mathbb{E} \sup_{f \in \mathcal F} \bigg|  \sum_{i \in \cup_{k \in \mathcal{K}} B_k}  \sigma_i \frac{(f-f^{*})(X_i)}{C_{K,r,\kappa}} \bigg|\leqslant \frac{r_2(\gamma,\kappa\rho^*)}{\sqrt{C_{K,r,\kappa}}}\frac{\gamma |\cK| N}{K}\leqslant \frac{\gamma |\cK| N}{K}\enspace.
\]
Plugging this bound in \eqref{eq:ConcB} yields, with probability larger than $1-e^{-|\cK|/288}$
\[
\sup_{f \in \mathcal F } \sum_{k \in \mathcal{K}} \bigg( \phi(2C_{K,r,\kappa}^{-1} | G_f(W_k)|) -  \mathbb{E} \phi(2C_{K,r,\kappa}^{-1} | G_f(W_k)|) \bigg)\leqslant |\mathcal{K}|\bigg(\frac{1}{24}+\frac{8L\gamma}{\theta }\bigg)=\frac{|\cK|}{12}\enspace.
\]
Plugging this inequality into \eqref{res::1} shows that, with probability at least $1-e^{-|\cK|/288}$,
\[
z(f)\geqslant \frac{7|\cK|}8\enspace.
\]
As $K\geqslant 7|\cO|/3$, $|\cK|\geqslant K-|\cO|\geqslant 4K/7$, hence, $z(f)\geqslant K/2$ holds with probability at least $1-e^{-K/504}$. Since it has to hold for any $\kappa$ in $\{ 1,2 \}$, the final probablity is $1-2e^{-K/504}$.
\end{proof}

\section{Proof Theorem~\ref{theo:main_without_bernstein_cond}} \label{sec:MOM_wtb}
The proof is very similar to the one of Theorem \ref{main:theo}. We only present the different arguments we use coming from the localization with the excess risk. The proof is split into two parts. First we identify an event $\bar{\Omega}_K$ in the same way is $\Omega_K $ in~\eqref{omega_mom} where the $L_2$-localization is replaced by the excess risk localization. For $\kappa \in \{1,2\}$ let $\cB_{\kappa} = \{  f \in E: \: P\cL_f \leq \bar{r}^2(\gamma, \kappa \rho^*), \; \|f-f^*\| \leq  \ \kappa \rho^*\}$ and 
\begin{align*} 
& \bar{\Omega}_K = \bigg\{ \forall \kappa \in \{1,2\}, \forall f \in F\cap \cB_{\kappa}, \sum_{k=1}^K I \big\{ |(P_{B_k} - P) \cL_f |  \leq \frac{1}{20}\bar{r}^2(\gamma,2\rho^*) \big\} \geq K/2 \bigg\} \\
\end{align*} 
Let us us the following notations,
\begin{align*}
\lambda = \frac{11\bar r^2(\gamma,2\rho^*)}{ 40 \rho^*}, \quad \hat f = \hat{f}_{K}^{\lambda} \quad \mbox{and}  \quad \gamma = 1/3840L 
\end{align*}
Finally recal that the complexity parameter is defined as
\begin{equation*}
 \bar{r}(\gamma,\rho)  = \inf \bigg\{ r > 0: \max \bigg(\frac{E(r,\rho)}{\gamma}, \sqrt{384000}V_K(r,\rho) \bigg)\leq r^2 \bigg\}
\end{equation*}
where
\begin{align*}
& E(r,\rho) = \sup_{J \subset \mathcal{I}:|J|\geq N/2}\mathbb{E}\sup_{f \in F: P\cL_f \leq r^2, \; \|f-f^*\| \leq \rho} \bigg | \frac{1}{ |J|}\sum_{i \in J} \sigma_i(f-f^{*})(X_i)   \bigg |\\
& V_K(r,\rho) =\max_{i\in\cI}\sup_{f \in F: P\cL_f \leq r^2, \; \|f-f^*\| \leq \rho}\left(\sqrt{\mathbb{V}ar_{P_i}(\cL_f)}\right) \sqrt{\frac{K}{N}}
\end{align*}
First, we show that on the event $\bar{\Omega}_K$, $P\cL_{\hat f} \leq \bar{r}^2(\gamma, 2\rho^*)$ and $\|f-f^*\| \leq 2\rho^*$. Then we will control the probability of $\bar{\Omega}_K$.

\begin{Lemma}
	Grant Assumptions~\ref{assum:lip} and~\ref{assum:convex}. Let $\rho^*$ satisfy the sparsity equation from Definition~\ref{def:SEMOM2}. On the event $\bar{\Omega}_K$, $P\cL_{\hat f} \leq \bar{r}^2(\gamma, 2\rho^*)$ and $\|f-f^*\| \leq 2\rho^*$.
\end{Lemma}
\begin{proof}
	Let $f \in F \backslash \cB_{\kappa}$. From Lemma 6 in \cite{ChiLecLer:2018} there exist $f_0 \in F$ and $\alpha >0$ such that $f-f^* = \alpha(f_0-f^*)$ and $f_0 \in \partial \cB_{\kappa}$. By definition of $\cB_{\kappa}$, either 1)$P\cL_{f_0} = \bar{r}^2(\gamma, \kappa\rho^*)$ and $\|f_0-f^*\| \leq \kappa \rho^*$ or 2) $P\cL_{f_0} \leq \bar{r}^2(\gamma, \kappa \rho^*)$ and $\|f_0-f^*\| = \kappa \rho^*$.\\
	
	Assume that $P\cL_{f_0} = \bar{r}^2(\gamma, \kappa \rho^*)$ and $\|f_0-f^*\| \leq \kappa \rho^*$. On $\bar{\Omega}_K$, there exist at least $K/2$ blocks $B_k$ such that $P_{B_k} \cL_{f_0} \geq P\cL_{f_0} - (1/20) \bar{r}^2(\gamma,\kappa \rho^*) = (19/20)\bar{r}^2(\gamma,\kappa\rho^*)$. It follows that, on at least $K/2$ blocks $B_k$ 
	\begin{equation} \label{eq1_wtb}
	P_{B_k} \cL_{f}^{\lambda} \geq \alpha P_{B_k} \cL_{f_0}^{\lambda} = \alpha \big(P_{B_k} \cL_{f_0} + \lambda (\|f_0\| - \|f^*\|) \big) \geq (19/20) \bar{r}^2(\gamma,\kappa \rho^*) -  11 \kappa  \bar{r}^2(\gamma,2 \rho^*)/40
	\end{equation}
	
	Assume that $P\cL_{f_0} \leq \bar{r}^2(\gamma, \kappa \rho^*)$ and $\|f_0-f^*\| = \kappa \rho^*$. From the sparsity equation defined in Definition~\ref{def:SEMOM2} we get $\|f_0\| - \|f^*\| \geq 7\kappa \rho^* /10$. And on more than $K/2$ blocks $B_k$ 
	
	\begin{equation} \label{eq2_wtb}
	P_{B_k} \cL_{f}^{\lambda} \geq  -(1/20) \bar{r}^2(\gamma,\kappa \rho^*) + 7\lambda \kappa \rho^* /10 = -(1/20) \bar{r}^2(\gamma,\kappa \rho^*) + 77 \kappa \bar{r}^2(\gamma,2 \rho^*)/ 400
	\end{equation}
	
	Now let us consider $f \in F \cap \cB_{\kappa}$. On $\bar{\Omega}_K$, there exist at least $K/2$ blocks $B_k$ such that 
	\begin{equation} \label{eq3_wtb}
	P_{B_k} \cL_{f}^{\lambda} \geq  -(1/20) \bar{r}^2(\gamma,\kappa \rho^*) - \lambda\kappa \rho^* = -(1/20)  \bar{r}^2(\gamma,\kappa \rho^*) - 11 \kappa \bar{r}^2(\gamma,2 \rho^*)/40
	\end{equation}

	As Equations~\eqref{eq1_wtb},~\eqref{eq2_wtb} and ~\eqref{eq3_wtb} hold for more than $K/2$ blocks it follows for $\kappa = 1$ that
	\begin{equation}\label{eq4_wtb}
	\sup_{f\in F} \MOM{K}{\ell_{f^{*}}-\ell_f} + \lambda (\|f^*\| - \|f\|) \leq (13/40)\bar{r}^2(\gamma,2 \rho^*) \enspace.
	\end{equation} 
	From Equations~\eqref{eq1_wtb},~\eqref{eq2_wtb} and ~\eqref{eq3_wtb} with $\kappa = 2$ we get 
	
	\begin{equation}\label{eq5_wtb}
	\sup_{f \in  F \backslash \cB_{2}}   \hspace{0.2cm} \MOM{K}{\ell_{f^{*}}-\ell_f} + \lambda(\|f^*\| -\|f\|) < (13/40)\bar{r}^2(\gamma,2 \rho^*) \enspace.
	\end{equation}
	From Equations~\eqref{eq4_wtb} and~\eqref{eq5_wtb} and a slight modification of Lemma~\ref{lem:Basic} it easy to see that on $\bar{\Omega}_K$, $P\cL_{\hat f} \leq \bar{r}^2(\gamma,2\rho^*)$ and $\|f-f^*\| \leq \rho^*$.
\end{proof}
\begin{Proposition} \label{prop_wtb}
	Grant Assumptions~\ref{assum:lip},~\ref{assum:convex} and~\ref{assum:moments2}. Then $\bar{\Omega}_K$ holds with probability larger than $1- 2\exp(-cK)$
\end{Proposition}

\textit{Sketch of proof.} The proof of Proposition~\ref{prop_wtb} follows the same line as the one of Proposition~\ref{lem:close}. Let us precise the main differences. For all $f\in F \cap \cB_{\kappa}$ we set, $z^\prime(f) = \sum_{k=1}^K I\{|G_f(W_k)|\leq (1/20)\bar{r}^2(\gamma,\kappa \rho^*) \}$ where $G_f(W_k)$ is the same quantity as in the proof of  Proposition~\ref{lem:close}. Let us consider the contraction $\phi$ introduced in Proposition~\ref{lem:close}. By definition of $V_K(r)$ and $\bar{r}^2(\gamma,\kappa \rho^*)$ we have
\begin{align*}
\mathbb{E}\phi(40 & | G_f(W_k)|/\bar{r}^2(\gamma,\kappa \rho^*))  \leq \mathbb{P} \bigg( |G_f(W_k)| \geq \frac{\bar{r}^2(\gamma,\kappa \rho^*)}{40} \bigg) \leq \frac{(40)^2}{\bar{r}^4(\gamma,\kappa \rho^*)} \mathbb{E}G_f(W_k)^2   \\
&=  \frac{(40)^2}{\bar{r}^4(\gamma,\kappa \rho^*)} \mathbb{V}ar (P_{B_k}\cL_f)   \leq  \frac{(40)^2 K^2}{\bar{r}^4(\gamma,\kappa \rho^*)N^2} \sum_{i \in B_k} \mathbb{V}ar_{P_i}(\cL_f) \\
& \leq \frac{(40)^2K}{\bar{r}^4(\gamma,\kappa \rho^*)N} \sup\{\mathbb{V}ar_{P_i}(\cL_f):f\in F \cap \cB_{\kappa}, i\in\cI\} \leq 1/24\enspace.
\end{align*}

Using Mc Diarmid's inequality, the Gin{\'e}-Zinn symmetrization argument and the contraction lemma twice and the Lipschitz property of the loss function,  such as in the proof of 	Proposition~\ref{lem:close}, we obtain for all $x>0$, with probability larger than $1-\exp(- |\cK|/288)$, for all $f\in\cF^\prime$,
\begin{equation}\label{eq:concentration_McDirmid_without_ass}
z'(f)\geq 11|\cK|/12 -\frac{160LK}{\theta N} \E \sup_{f\in F \cap \cB_{\kappa}} \frac{1}{\bar{r}^2(\gamma,\kappa \rho^*)}\left|\sum_{i\in\cup_{k\in\cK}B_k} \sigma_i (f-f^*)(X_i)\right|. 
\end{equation} 
From the definition of $\bar{r}^2(\gamma,\kappa \rho^*)$ it follows that $\E \sup_{f\in F \cap \cB_{\kappa}} \left|\sum_{i\in\cup_{k\in\cK}B_k} \sigma_i (f-f^*)(X_i)\right| \leq \gamma \bar{r}^2(\gamma,\kappa \rho^*) $ and $z'(f) \geq |\cK|(11/12 - 160L^2\gamma) = 7 |\cK|/8 $. The rest of the proof is totally similar.

\subsection{Proof of Theorem~\ref{theo:bernstein_kappa}}\label{ProofBernsteinKappa}
	From Assumption~\ref{assum:lip}, it holds $V_K(r) \leq L V_K'(r)$, where for all $r>0$,  
	\[
	V_K'(r) =  \sqrt{K/N}\max_{i\in\cI}\sup_{f \in F: P\cL_f \leq r^2, \; \|f-f^*\| \leq \rho} \|f-f^*\|_{L_2}\enspace.
	\]
By Assumption~\ref{ass:bernstein_kappa},
	\begin{equation*}
	\sqrt{c} V_K\bigg( \sqrt{384000}L \sqrt{\frac{\bar{A} K}{N}}, 2\rho^* \bigg) \leq 	384000L^2 \frac{\bar{A} K}{N}\enspace.
	\end{equation*}
	From the definition of $ r_2^2(\gamma,2\rho^*)$ and Assumption~\ref{ass:bernstein_kappa}, it follows
	\begin{equation*}
	\frac{1}{\gamma} E \bigg( \frac{ r_2(\gamma/ \bar A, 2\rho^*)}{\sqrt{\bar{A}}}  \bigg) \leq \frac{r_2^2(\gamma/ \bar A, 2\rho^*)}{\bar{A}}\enspace.
	\end{equation*}
	Hence, $\bar{r}^2(\gamma, 2\rho^*) \leq \max \big( r_2^2(\gamma/\bar{A}, 2\rho^*)/\sqrt{\bar{A}}, 384000 L^2\bar A K/N \big)$ and the proof is complete.

\end{document}